\documentclass[
final
]{dmtcs-episciences}

\usepackage[utf8]{inputenc}
\usepackage{subfigure}

\usepackage[round]{natbib}

\newtheorem{theorem}{Theorem}[section]
\newtheorem{lemma}[theorem]{Lemma}
\newtheorem{corollary}[theorem]{Corollary}
\newtheorem{conjecture}{\bf Conjecture}
\newtheorem{fact}{\bf Fact}
\newtheorem{claim}{\bf Claim}

\author{Aijun Dong\affiliationmark{1}\thanks{This work was supported by the National
Natural Science Foundation of China (Grant No. 61773015). It was
also supported by China Postdoctoral Science Foundation Funded
Project (Grant No.2014M561909); the Nature Science Foundation of
Shandong Province of China (Grant No. ZR2014AM028, ZR2017LEM014)
.}
  \and Jianliang Wu\affiliationmark{2}}
\title[Equitable Coloring and Equitable Choosability]{Equitable Coloring and Equitable
Choosability of Planar Graphs without chordal 4- and 6-Cycles}
\affiliation{
  School of Science, Shandong Jiaotong University,
Jinan, 250357, P. R. China\\
  School of Mathematics, Shandong University, Jinan,
250100, P. R. China} 
\received{2018-06-05}
\revised{2019-03-06}
\accepted{2019-04-19}

\begin{document}
\publicationdetails{21}{2019}{3}{16}{4566}
\maketitle
\begin{abstract}
A graph $G$ is equitably $k$-choosable if, for any given
$k$-uniform list assignment $L$, $G$ is $L$-colorable and each
color appears on at most $\lceil\frac{|V(G)|}{k}\rceil$ vertices.
A graph is equitably $k$-colorable if the vertex set $V(G)$ can be
partitioned into $k$ independent subsets $V_1$, $V_2$, $\cdots$,
$V_k$ such that $||V_i|-|V_j||\leq 1$ for $1\leq i, j\leq k$. In
this paper, we prove that if $G$ is a planar graph without chordal
$4$- and $6$-cycles, then $G$ is equitably $k$-colorable and
equitably $k$-choosable where $k\geq\max\{\Delta(G), 7\}$.
\end{abstract}

\section{Introduction}
The terminology and notation used but undefined in this paper can
be found in~\cite{bondy}. Let $G=(V,E)$ be a graph. We use $V(G)$,
$E(G)$, $\Delta(G)$ and $\delta(G)$ to denote the vertex set, edge
set, maximum degree and, minimum degree of $G$, respectively.
Particularly, we use $F(G)$ to denote the face set of $G$ when $G$
is a plane graph. Let $d_G(x)$ or simply $d(x)$, denote the degree
of a vertex (resp. face) $x$ in $G$. A vertex (resp. face) $x$ is
called a $k$-$vertex$ (resp. $k$-$face$), $k^+$-$vertex$ (resp.
$k^+$-$face$), $k^-$-$vertex$ or $k^{--}$-$vertex$, if $d(x)=k$,
$d(x)\geq k$, $2\leq d(x)\leq k$, or $1\leq d(x)\leq k$. We use
$(d_1, d_2, \cdots, d_n)$ to denote a face $f$ if $d_1, d_2,
\cdots, d_n$ are the degrees of vertices incident with the face
$f$ where $3\leq n\leq 5$. Let $\delta(f)$ denote the minimal
degree of vertices incident with $f$. In the following, let
$f_i(v)$ denote the number of $i$-faces incident with $v$ for each
$v\in V(G)$. Let $n_i(f)$ denote the number of $i$-vertices which
are incident with $f$. A graph $G$ is $k$-degenerate if every
subgraph of $G$ has a vertex of degree at most $k$. A cycle $C$ of
length $k$ is called a $k$-$cycle$. Moreover, if there exists an
edge $xy\in E(G)-E(C)$ and $x$, $y\in V(C)$, then the cycle $C$ is
called a chordal $k$-$cycle$.

A proper $k$-coloring of a graph $G$ is a mapping $\pi$ from the
vertex set $V(G)$ to the set of colors $\{1,2,\cdots,k\}$ such
that $\pi(x)\neq\pi(y)$ for every edge $xy\in E(G)$. A graph $G$
is equitably $k$-colorable if $G$ has a proper $k$-coloring such
that the sizes of the color classes differ by at most 1. The
equitable chromatic number of $G$, denoted by $\chi_e(G)$, is the
smallest integer $k$ such that $G$ is equitably $k$-colorable. The
equitable chromatic threshold of $G$, denoted by $\chi^*_e(G)$, is
the smallest integer $k$ such that $G$ is equitably $l$-colorable
for every $l\geq k$. It is obvious that $\chi_e(G)\leq
\chi^*_e(G)$ for any graph $G$. However these two parameters may
not be the same. For example, if $K_{2n+1, 2n+1}$ ($n$ is a
positive integer) is a complete bipartite graph, then
$\chi_e(K_{2n+1, 2n+1})=2$, $\chi^*_e(K_{2n+1, 2n+1})=2n+2$.

In many applications of graph coloring, it is desirable that the
color classes are not too large. For example, when using a
coloring model to find an optimal final exam schedule, one would
like to have approximately equal number of final exams in each
time slot because the whole exam period should be as short as
possible and the number of classrooms available is limited.
Recently, ~\cite{pemma}, ~\cite{jans} used equitable colorings to
derive deviation bounds for sums of dependent random variables
that exhibit limited dependence. In all of these applications, the
fewer colors we use, the better the deviation bound is. Equitable
coloring has a well-known property that restricts the size of each
color class by its definition.

In 1970, ~\cite{hajs} proved that $\chi^*_e(G)\leq \Delta(G)+1$
for any graph $G$. This bound is sharp as the example of $K_{2n+1,
2n+1}$ shows. In 1973, ~\cite{meye} introduced the notion of
equitable coloring and made the following conjecture.

\begin{conjecture}\label{meyerconj}  If $G$ is a connected graph which is neither a
complete graph nor odd cycle, then $\chi_e(G)\leq \Delta(G)$.
\end{conjecture}

In 1994, ~\cite{chenlw} put forth the following conjecture.

\begin{conjecture}\label{chenconj}  For any connected graph $G$, if it is different
from a complete graph, a complete bipartite graph and an odd
cycle, then $\chi^*_e(G)\leq \Delta(G)$.
\end{conjecture}

~\cite{chenlw, chenl} proved Conjecture~\ref{chenconj} for graphs
with $\Delta(G)\leq 3$ or $\Delta(G)\geq \frac{|V(G)|}{2}$.
Recently, ~\cite{chenY} improved the former result and confirmed
the Conjecture~\ref{chenconj} for graphs with $\Delta(G)\geq
\frac{|V(G)|}{3}+1$. ~\cite{yapz1, yapz2} showed that
Conjecture~\ref{chenconj} holds for planar graphs with
$\Delta(G)\geq 13$. Recently, ~\cite{Nakprasit} confirmed the
Conjecture~\ref{chenconj} for planar graphs with $\Delta(G)\geq
9$. ~\cite{lihw} verified $\chi^*_e(G)\leq \Delta(G)$ for
bipartite graphs other than complete bipartite graphs.
~\cite{wanz} proved Conjecture~\ref{chenconj} for line graphs, and
~\cite{kosn1, kosn2} proved it for graphs with low average degree,
and $d$-degenerate graphs with $\Delta(G)\geq
14d+1$.~\cite{yanwang} showed that Conjecture~\ref{chenconj} holds
for Kronecker products of complete multipartite graphs and
complete graphs. ~\cite{jianl}, ~\cite{luo} confirmed
Conjecture~\ref{chenconj} for some planar graphs with large girth,
respectively. Recently, ~\cite{qiong}, ~\cite{jun1}, ~\cite{dong1,
dong2, dong3}, ~\cite{Nakprasit2} confirmed
Conjecture~\ref{chenconj} for some planar graphs with some
forbidden cycles, respectively. ~\cite{zhang}, ~\cite{jun2}
verified the Conjecture~\ref{chenconj} for some series-parallel
graphs and outerplanar graphs, respectively.

For a graph $G$ and a list assignment $L$ assigned to each vertex
$v\in V(G)$ a set $L(v)$ of acceptable colors, an $L$-coloring of
$G$ is a proper vertex coloring such that for every $v\in V(G)$
the color on $v$ belongs to $L(v)$. A list assignment $L$ for $G$
is $k$-$uniform$ if $|L(v)|=k$ for all $v\in V(G)$. A graph $G$ is
equitably $k$-choosble if, for any $k$-uniform list assignment
$L$, $G$ is $L$-colorable and each color appears on at most
$\lceil\frac{|V(G)|}{k}\rceil$ vertices.

In 2003, Kostochka, Pelsmajer and West investigated the equitable
list coloring of graphs. They proposed the following conjectures
in~\cite{kostochka}.

\begin{conjecture}\label{kostochconj1} Every graph $G$ is equitably $k$-choosable
whenever $k>\Delta(G)$.
\end{conjecture}

\begin{conjecture}\label{kostochconj2} If $G$ is a connected graph with maximum
degree at least $3$, then $G$ is equitably $\Delta(G)$-choosable,
unless $G$ is a complete graph or is $K_{k,k}$ for some odd $k$.
\end{conjecture}

It has been proved that Conjecture~\ref{kostochconj1} holds for
graphs with $\Delta(G)\leq3$ in ~\cite{pelsmajer, wang} and graphs
with $\Delta(G)\leq7$ in~\cite{kosn3}. Kostochka, Pelsmajer and
West proved that a graph $G$ is equitably $k$-choosable if either
$G\neq K_{k+1}, K_{k,k}$ (with $k$ odd in $K_{k,k}$) and
$k\geq\max\{\Delta, \frac{|V(G)|}{2}\}$, or $G$ is a connected
interval graph and $k\geq\Delta(G)$ or $G$ is a $2$-degenerate
graph and $k\geq\max\{\Delta(G),5\}$ in~\cite{kostochka}.
Pelsmajer proved that every graph is equitably $k$-choosable for
any $k\geq\frac{\Delta(G)(\Delta(G)-1)}{2}+2$ in~\cite{pelsmajer}.
Bu and his collaborators have established a series results for
Conjecture~\ref{kostochconj2} in class of planar graph as follows
~\cite{qiong,jun1,jun2,jun3}. Zhang and Wu proved
Conjecture~\ref{kostochconj2} for series-parallel graphs
in~\cite{zhang}. Some improved results on planar graphs were
obtained in~\cite{dong1}, ~\cite{dong2} and~\cite{dong3}.

In this paper, we improve the result in~\cite{qiong} and confirm
the Conjecture~\ref{chenconj}, Conjecture~\ref{kostochconj2} for
some planar graphs in which $4$- and $6$-cycles are allowed to
exist, which shows that if $G$ is a planar graph without chordal
$4$- and $6$-cycles, then $G$ is equitably $k$-colorable and
equitably $k$-choosable where $k\geq\max\{\Delta(G), 7\}$.

\section{Planar graphs without chordal 4- and 6-cycles}

First let us introduce some lemmas.

\vspace{-0.3cm}

\begin{lemma}\label{lemma1} Let $G$ be a planar
graph without chordal $4$- or $6$-cycles. Then in $G$, there is no
$3$-cycle adjacent to a $3$-cycle, nor a $4$-cycle adjacent to two
$3$-cycles. Furthermore, if $\delta(G)\geq3$, then there is no
$3$-cycle adjacent to a $5$-cycle, nor a $4$-cycle adjacent to a
$4$-cycle.
\end{lemma}

By Lemma~\ref{lemma1}, we have the following lemma.

\begin{lemma}\label{lemma01} Let $G$ be a planar graph with $\delta(G)\geq3$ and $f$ be a $3$-face
which is incident with a $3$-vertex in $G$. Then $f$ is adjacent
to at least one $6^+$-face.
\end{lemma}

\begin{lemma}\label{lemma2} Let $G$ be a planar graph without chordal $4$- and $6$-cycles.
If $\delta(G)\geq 4$, then $G$ contains the configuration $H$
depicted in Figure 1.
\end{lemma}

\begin{proof} Suppose to the contrary that $G$ does not contain
the configuration $H$ depicted in Figure 1, i.e. none of the
$(4,4,4)$-faces is adjacent to a $(4,4,4,4)$-face.

\begin{figure}[htbp]
  \begin{center}
    \includegraphics[width=2.5cm]{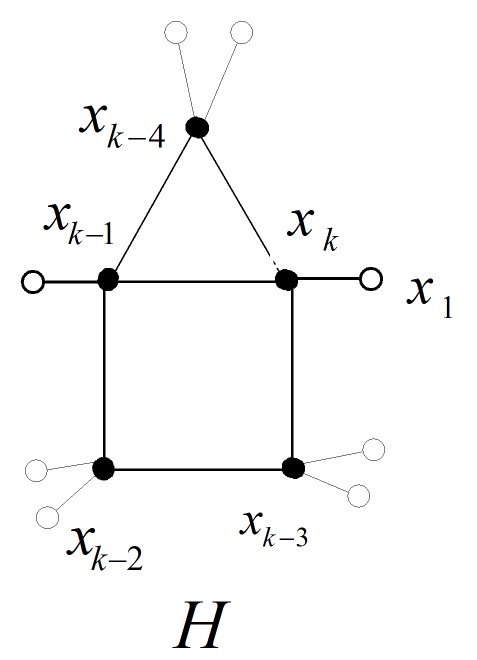}
    \caption{}
  \end{center}
\end{figure}

By Euler's formula, we have
\begin{eqnarray}\label{formula0}
\sum_{v\in V(G)}(2d(v)-6)+\sum_{f\in
F(G)}(d(f)-6)=-6(|V|-|E|+|F|)=-12.
\end{eqnarray}

\vspace{2mm} Define an initial charge function $w$ on $V(G)\cup
F(G)$ by setting $w(v)=2d(v)-6$ if $v\in V(G)$ and $w(f)=d(f)-6$
if $f\in F(G)$, then $\sum_{x\in V(G)\cup F(G)}w(x)=-12$ by
Equation~(\ref{formula0}). Now redistribute the charges according
to the following discharging rules.

\noindent $D1.$ If $f$ is a $3$-face incident with a vertex $v$,
then $v$ gives $1$ to $f$ if $d(v)=4$ and $f$ is a $(4,4,4)$-face,
$v$ gives $\frac{3}{4}$ if $d(v)=4$ and $f$ is a $3$-face of
another type, and $v$ gives $\frac{3}{2}$ if $d(v)\geq5$.

\noindent $D2.$ If $f$ is a $4$-face incident with a vertex $v$,
then $v$ gives $\frac{1}{2}$ to $f$ if $d(v)=4$ and $f$ is a
$(4,4,4,4)$-face, $v$ gives $\frac{2}{5}$ if $d(v)=4$ and $f$ is a
$4$-face of another type, and $v$ gives $\frac{4}{5}$ if
$d(v)\geq5$.

\noindent $D3.$ Transfer $\frac{1}{5}$ from each vertex $v$ to the
$5$-face which is incident with $v$.

\vspace{2mm} Let the new charge of each element $x\in V(G)\cup
F(G)$ be $w'(x)$. In the following, we will show that $\sum_{x\in
V(G)\cup F(G)}w'(x)\geq0$, a contradiction to
Equation~(\ref{formula0}). This
will complete the proof.\\

Consider any vertex $v\in V(G)$, \textbf{suppose $d(v)=4$.} Then
$w(v)=2$, $f_3(v)\leq 2$ by Lemma~\ref{lemma1}.

First, we assume that $f_3(v)=2$. Then $f_4(v)=0$ and $f_5(v)=0$
by Lemma~\ref{lemma1}. Thus $w'(v)\geq 2-1\times2=0$ by $D1$.

Now we assume that $f_3(v)=1$. Then $f_4(v)\leq 2$ by
Lemma~\ref{lemma1}. If $f_4(v)=2$, then $f_5(v)\leq 1$. Since $G$
does not contain the configuration $H$ depicted in Figure 1, thus
$w'(v)\geq 2-1-\frac{2}{5}\times2-\frac{1}{5}=0$ or
$w'(v)\geq2-\frac{3}{4}-\frac{1}{2}\times2-\frac{1}{5}=\frac{1}{20}>0$
by $D1$, $D2$ and $D3$. If $f_4(v)\leq1$, then $f_5(v)\leq 1$ by
Lemma~\ref{lemma1}. Thus $w'(v)\geq
2-1-\frac{1}{2}-\frac{1}{5}=\frac{3}{10}>0$ by $D1$, $D2$ and
$D3$.

Now we assume that $f_3(v)=0$. Then $f_4(v)\leq 2$, $f_5(v)\leq4$
by Lemma~\ref{lemma1}. Thus $w'(v)>
2-\frac{1}{2}\times2-\frac{1}{5}\times4=\frac{1}{5}>0$ by $D2$ and
$D3$.\\

\textbf{Suppose $d(v)=5$.} Then $w(v)=4$, $f_3(v)\leq2$ by
Lemma~\ref{lemma1}. If $f_3(v)=2$, then $f_4(v)\leq1$ and
$f_5(v)=0$ by Lemma~\ref{lemma1}. Thus $w'(v)\geq
4-\frac{3}{2}\times2-\frac{4}{5}=\frac{1}{5}>0$ by $D1$ and $D2$.
If $f_3(v)=1$, then $f_4(v)\leq2$ and $f_5(v)\leq2$ by
Lemma~\ref{lemma1}. Thus $w'(v)\geq
4-\frac{3}{2}-\frac{4}{5}\times2-\frac{1}{5}\times2=\frac{1}{2}>0$
by $D1$, $D2$ and $D3$. If $f_3(v)=0$, then $f_4(v)\leq2$ and
$f_5(v)\leq5$ by Lemma~\ref{lemma1}. Thus $w'(v)>
4-\frac{4}{5}\times2-\frac{1}{5}\times5=\frac{7}{5}>0$ by $D2$ and
$D3$.\\

\textbf{Suppose $d(v)\geq6$.} Then $w(v)=2d(v)-6$, $f_4(v)\leq
d(v)-2f_3(v)$, $f_5(v)\leq d(v)-2f_3(v)$ by Lemma~\ref{lemma1}. So
$w'(v)\geq
2d(v)-6-\frac{3}{2}f_3(v)-\frac{4}{5}f_4(v)-\frac{1}{5}f_5(v)\geq
d(v)-6+\frac{1}{2}f_3(v)\geq d(v)-6\geq0$ by $D1$, $D2$ and
$D3$.\\

Consider any face $f\in F(G)$, \textbf{suppose $d(f)=3$.} Then
$w(f)=-3$. If $f$ is a $(4,4,4)$-face, then $w'(f)=-3+1\times3=0$
by $D1$. Otherwise,
$w'(f)\geq -3+\frac{3}{4}+\frac{3}{4}+\frac{3}{2}=0$  by $D1$.\\

\textbf{Suppose $d(f)=4$.} Then $w(f)=-2$. If $f$ is a
$(4,4,4,4)$-face, we have that $w'(f)\geq -2+\frac{1}{2}\times4=0$
by $D2$. Otherwise, $w'(f)\geq
-2+\frac{2}{5}\times3+\frac{4}{5}=0$ by $D2$.\\

\textbf{Suppose $d(f)=5$.} Then $w(f)=-1$. We have $w'(f)\geq
-1+\frac{1}{5}\times5=0$ by $D3$.\\

\textbf{Suppose $d(f)\geq6$.} Then $w'(f)=w(f)\geq0$.
\end{proof}

\begin{lemma}\label{jun1} (\cite{jun1}) Let $S=\{x_1, x_2, \cdots, x_k\}$ be a set of $k$
different vertices in $G$ such that $G-S$ has an equitable
$k$-coloring. If $|N_G(x_i)-S|\leq k-i$ for $1\leq i\leq k$, then
$G$ has an equitable $k$-coloring.
\end{lemma}

\begin{lemma}\label{kostochka1} (\cite{kostochka}) Let $G$ be a graph with
a $k$-uniform list assignment $L$. Let $S=\{x_1,
x_2,\cdots,x_k\}$, where $x_1, x_2,\cdots,x_k$ are distinct
vertices in $G$. If $G-S$ has an equitable $L$-coloring and
$|N_G(x_i)-S|\leq k-i$ for $1\leq i\leq k$, then $G$ has an
equitable $L$-coloring.
\end{lemma}

\begin{lemma}(\cite{borodin})\label{lemborodin} Every planar graph without adjacent triangles is
$4$-degenerate.
\end{lemma}

By Lemma~\ref{lemborodin}, we have the following corollary.

\begin{corollary}\label{cor4degenerate}
Let $G$ be a planar graph without chordal $4$-cycles. Then $G$ is
$4$-degenerate.
\end{corollary}

\begin{lemma}\label{lemma3} Let $G$ be a connected planar graph with order at least
$5$ and without chordal $4$- and $6$-cycles. If $\delta(G)\leq3$,
then $G$ has at least one of the configurations depicted in Figure
2.
\end{lemma}

\vspace{-0.3cm}

\begin{proof}
Suppose to the contrary that $G$ does not contain the
configurations $H_1$ $\ldots$ $H_{41}$ depicted in Figure 2.

\begin{figure}[htbp]
  \begin{center}
    \includegraphics[width=13.5cm]{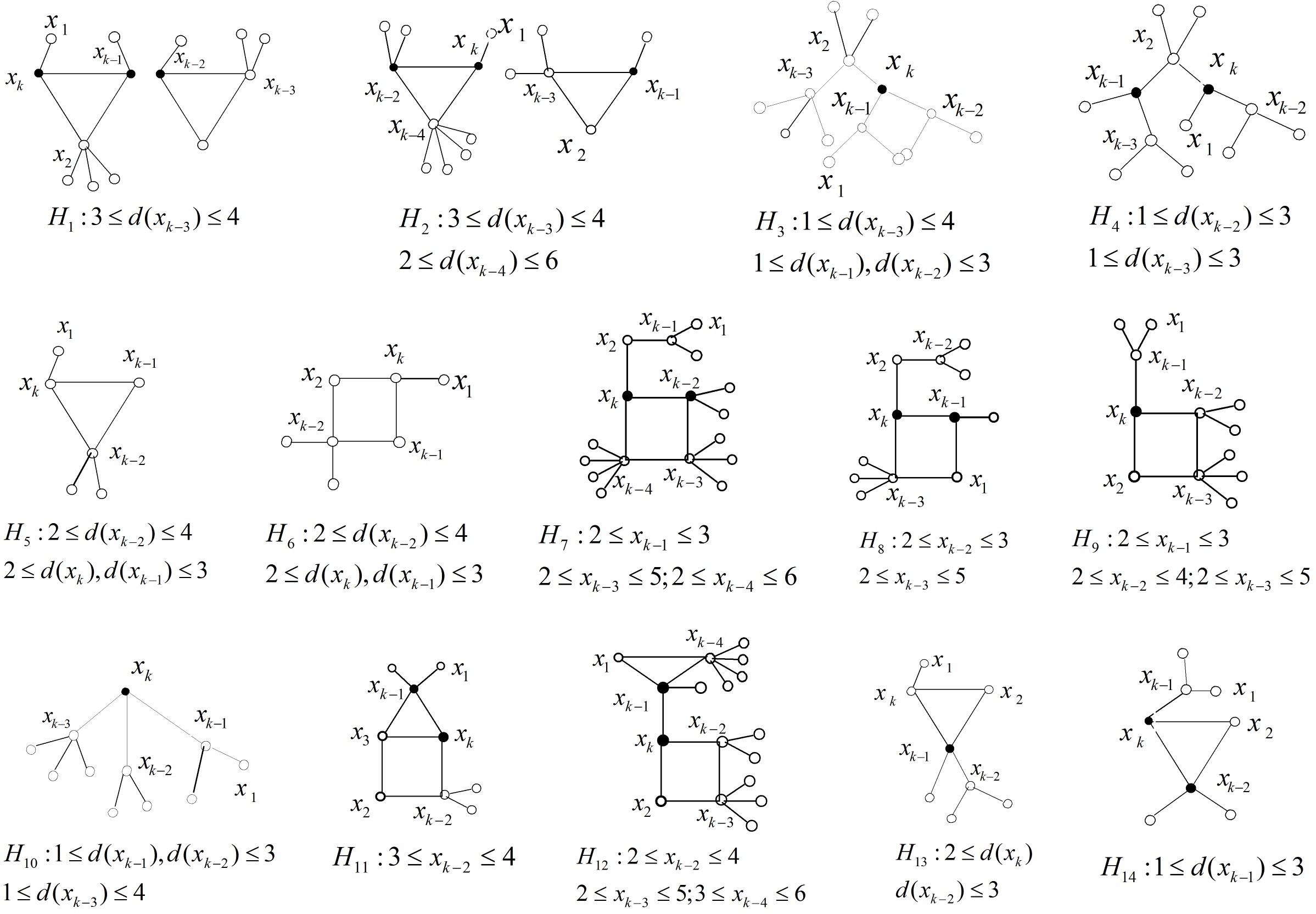}
  \end{center}
\end{figure}

\begin{figure}[htbp]
  \begin{center}
    \includegraphics[width=1\textwidth]{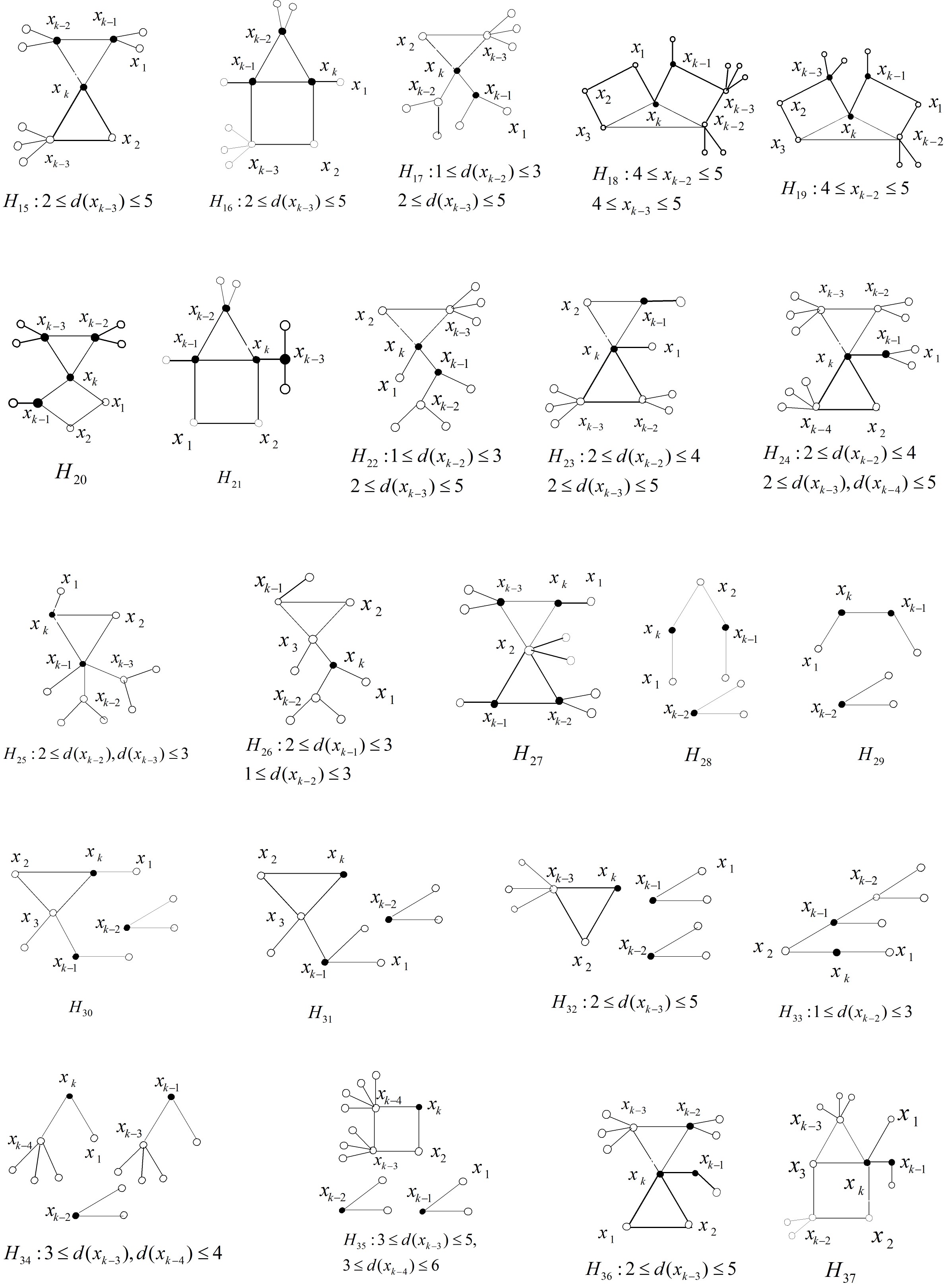}
  \end{center}
\end{figure}

\begin{figure}[htbp]
  \begin{center}
    \includegraphics[width=13.5cm]{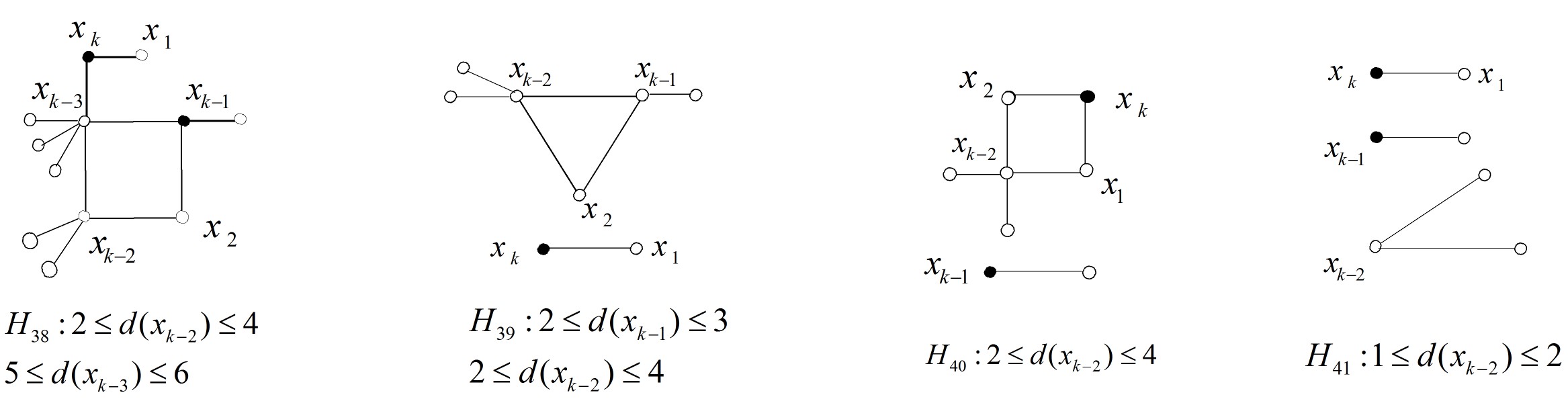}
    \caption{}
  \end{center}
\end{figure}

\emph{Each configuration depicted in Figure 2 is such that: $~$
(1) the vertices labelled $x_k, x_{k-1}, x_{k-2}$ are distinct and
the other vertices may coincide if they have the same degree and
multiple edges cannot be resulted in; $~$(2) solid vertices have
no incident edges other than the ones shown; and $~$(3) except for
being specially pointed, the degree of a hollow vertex may be any
integer from $[d, \Delta(G)]$, where $d$ is the number of edges
incident with the hollow vertex shown in the configuration; $~$(4)
the order of the vertices on the boundary of a $4$-face can be
rearranged except for the vertex which is also adjacent to other
labelled vertex that is not on
the boundary of the $4$-face.}\\

A face is said to be a $special$ $face$ if it is a
$(3,3,5^+)$-face, $(3,4,4)$-face, $(3,4,5)$-face or a
$(3,4,6)$-faces. In the following, we call a $3$-vertex a
$special$ $3$-$vertex$ if it is incident with a special face,
otherwise, it is called a $simple$ $3$-$vertex$.

Since $G$ contains neither $H_1$ nor $H_2$, we obtain the
following property.

\begin{claim}\label{claim1} There is at most one special face in
$G$.
\end{claim}

By Claim~\ref{claim1}, $G$ has at most two special $3$-vertices.
For convenience, let $n_3(v)$ denote the number of simple
$3$-vertices adjacent to $v$ for each $v\in V(G)$. Since $G$
contains neither $H_{3}$ nor $H_{4}$, we can conclude the
following properties.

\begin{claim}\label{claim2}  For each $v\in V(G)$ with $d(v)\geq4$, if
$v$ is adjacent to a simple $3$-vertex which is adjacent to two
$3^{--}$-vertices, then it is not adjacent to another
$4^{--}$-vertex.
\end{claim}

\begin{claim}\label{claim3} For any $v\in V(G)$ with $d(v)\geq4$, $v$ is adjacent to at
most one simple $3$-vertex which is adjacent to another
$3^{--}$-vertex.
\end{claim}

By Euler's formula $|V|-|E|+|F|=2$ and $\sum_{v\in
V(G)}d(v)=\sum_{f\in F(G)}d(f)=2|E|$, thus

\begin{eqnarray}\label{formula1}
\sum_{v\in V(G)}(3d(v)-10)+\sum_{f\in
F(G)}(2d(f)-10)=-10(|V|-|E|+|F|)=-20.
\end{eqnarray}

Define an initial charge function $w$ on $V(G)\cup F(G)$ by
setting $w(v)=3d(v)-10$ if $v\in V(G)$ and $w(f)=2d(f)-10$ if
$f\in F(G)$.

In the following, we divide the proof into four cases.

\vspace{0.3cm} \noindent {\bf Case 1.} $\delta(G)=3$.

Since $G$ does not contain the configuration $H_{5}$, $G$ has the
following property.

\begin{fact}\label{fact11} Any $3$-face in $G$ is a $(3, 3, 5^+)$-,
$(3,4^+,4^+)$- or $(4^+,4^+,4^+)$-face, i.e. there is no
$(3,3,4^-)$-face.
\end{fact}

 Since $G$ does not contain the configuration
$H_{6}$, $G$ has the following property.

\begin{fact}\label{fact12} Any $4$-face in $G$ is a $(3, 3,
5^+,5^+)$-, $(3,4^+,4^+,4^+)$- or $(4^+,4^+,4^+,4^+)$-face, i.e.
there is no $(3,3,3,3^+)$- or $(3,3,4,4^+)$-face.
\end{fact}

For convenience, if a face is a $(3,3,5,5^+)$- or
$(3,4,5^-,6^-)$-face, then we call it a $bad$ $face$. The
$3$-vertex which is incident with a bad face is said to be a $bad$
$3$-$vertex$. If a vertex $v$ is adjacent to a bad $3$-vertex $w$
and $v$ is not incident with the bad face $f$ which is incident
with the vertex $w$, then we say that $v$ is $weakly$ $incident$
with the bad face $f$.

Now redistribute the charge according to the following discharging
rules.

\begin{itemize}

 \item \textbf{$R1.$ Transfer $1$ from each $5^+$-vertex to every adjacent simple
$3$-vertex which is adjacent to exactly two $3^{--}$-vertices}.

\item \textbf{$R2.$ Transfer $\frac{1}{2}$ from each $4^+$-vertex
to every adjacent simple $3$-vertex which is adjacent to exactly
one $3^{--}$-vertex}.

\item  \textbf{$R3.$ Transfer $\frac{1}{3}$ from each $4^+$-vertex
to every adjacent simple $3$-vertex which is not adjacent to any
$3^{--}$-vertex}.

\item  \textbf{$R4.$ Transfer $\frac{1}{3}$ from each $6^+$-face
$f$ to every adjacent $3$-face and $4$-face via each common edge}.

\item  \textbf{$R5.$ If $f$ is a $4$-face incident with a vertex
$v$, then $v$ gives $\frac{1}{2}$ to $f$ if $d(v)=4$ and $f$ is a
$(3,4,5^-,5^-)$- or $(4,4,4,4^+)$-face, $\frac{1}{3}$ if $d(v)=4$
and $f$ is either a $(3,4,4^+,6^+)$- or a $(4,4^+,5^+,5^+)$-face;}

\textbf{$\frac{1}{2}$ if $d(v)=5$ and $f$ is a
$(5,5^+,5^+,5^+)$-face; $\frac{2}{3}$ if $d(v)=5$ and $f$ is a
$4$-face of another type};

\textbf{$1$ if $d(v)=6$ and $f$ is a $(3,3,6,6^+)$- or
$(3,4,6,4^+)$-face, $\frac{2}{3}$ if $d(v)=6$ and $f$ is a
$(3,6,5^+,5^+)$- or $(4^+,6,4^+,4^+)$-face};

\textbf{$\frac{4}{3}$ if $d(v)\geq 7$}.

\item  \textbf{$R6.$ If $f$ is a $3$-face incident with a vertex
$v$ with $d(v)=4$, then $v$ gives $\frac{2}{3}$ to $f$} if $f$ is
a $(3,4,4^+)$-face, $\frac{4}{3}$ if $f$ is a $(4,4,4)$-face, $1$
if $f$ is a $(4,4,5^+)$- or $(4,5,5^+)$-face, $0$ if $f$ is a
$(4,6^+,6^+)$-face;

\textbf{If $f$ is a $3$-face incident with a vertex $v$ with
$d(v)=5$, then $v$ gives $\frac{11}{6}$ to $f$ if $f$ is a
$(3,5,3^+)$-face, $2$ if $f$ is a $(4,4,5)$- or $(4,5,5)$-face,
$\frac{4}{3}$ if $f$ is a $(5,5,5)$-face, $1$ if $f$ is a
$(4,5,6^+)$-, $(5,5,6^+)$- or $(5,6^+,6^+)$-face};

\textbf{If $f$ is a $3$-face incident with a vertex $v$ with
$d(v)=6$, then $v$ gives $2$ to $f$};

\textbf{If $f$ is a $3$-face incident with a vertex $v$ with
$d(v)\geq7$, then $v$ gives $3$ to $f$ if $f$ is a $(3,3,7^+)$- or
$(3,4,7^+)$-face, $2$ if $f$ is a $(3,5^+,7^+)$- or
$(4^+,4^+,7^+)$-face}.

\item  \textbf{$R7.$ If $f$ is a bad face and $v$ is weakly
incident with $f$, then $v$ gives charge $\frac{1}{2}$ to $f$}.

\end{itemize}

\vspace{0.2cm}

In the following, let us check the new charge of each element $x$
for $x\in V(G)\cup F(G)$.

For convenience, we use $f_{k}^{\alpha}(v)$ (respectively,
$n_3^{\alpha}(v)$) to denote the number of $k$-faces
(respectively, $3$-vertices) which are incident with $v$ and
receive charge at least $\alpha$ from $v$ according to the
discharging rules.

By Claim~\ref{claim2}, Claim~\ref{claim3}, $R1$, $R2$ and $R3$, we
have the following fact.

\begin{fact}\label{fact13}
For each $v\in V(G)$, obviously, $n_3^{\frac{1}{2}}(v)\leq1$, and
if $n_3^{1}(v)\neq0$, then $n_3(v)=1$ and the degrees of other
neighbors of $v$ are at least 5.
\end{fact}

Since $G$ contains no configurations $H_{7}$ and $H_{8}$, thus the
following fact holds.

\begin{fact}\label{fact14} For
each $v\in V(G)$, $v$ is weakly incident with at most one bad
face. Furthermore, if $v$ is weakly incident with a bad face, then
$n_3(v)=1$.
\end{fact}

Let $v\in V(G)$. \textbf{Suppose $d(v)=3$}. Then $w(v)=-1$. Since
$G$ contains no configuration $H_9$, $v$ is not weakly incident
with any bad face. Since $G$ contains no configuration $H_{10}$,
$v$ is adjacent to at least one $5^+$-vertex or is adjacent to at
least two $4^+$-vertices. If $v$ is a simple $3$-vertex, then
$w'(v)=-1+1=0$ by $R1$, $w'(v)=-1+\frac{1}{2}\times2=0$ by $R2$ or
$w'(v)=-1+\frac{1}{3}\times3=0$ by $R3$. Otherwise, i.e. if $v$ is
a special $3$-vertex, then $w'(v)=w(v)=-1$.\\

\textbf{Suppose $d(v)=4$}. Then $w(v)=2$.

First, we assume that $v$ is weakly incident with a bad face.
Since $G$ contains no configuration $H_{11}$, we have
$f_3(v)\leq1$. Additionally, if $f_3^{\alpha}=1$, we have
$\alpha=0$ because $G$ contains no configuration $H_{12}$ and by
$R6$. By Lemma~\ref{lemma1}, we have $f_4(v)\leq1$. Clearly,
$w'(v)\geq 2-\frac{1}{2}-\frac{1}{2}-\frac{1}{2}=\frac{1}{2}>0$ by
Fact~\ref{fact14}, $R2$, $R5$ and $R7$.

Now we assume that $v$ is not weakly incident with a bad face.
Clearly, we have $f_3(v)\leq2$. For convenience, we divide the
proof into the following cases.\\

\textbf{Case $1.1$} $f_3(v)=2$. Then $n_3(v)\leq 1$, $f_4(v)=0$
for the reason that $G$ contains no configuration $H_{13}$, by
Fact 1 and Lemma~\ref{lemma1}. If $f_3(v)=2$, $n_3(v)=1$, then we
have that $f_3^{\frac{4}{3}}(v)=0$ and $n_3^{\frac{1}{2}}(v)=0$
for the reason that $G$ contains no configurations $H_{15}$,
$H_{14}$ and by $R6$, $R2$, $R1$. Clearly, $w'(v)\geq
2-\frac{2}{3}-1-\frac{1}{3}=0$ by $R6$ and $R3$. If $f_3(v)=2$,
$n_3(v)=0$ and $f_3^{\frac{4}{3}}(v)\neq0$, then we have that
$w'(v)\geq 2-\frac{4}{3}=\frac{2}{3}>0$ for the reason that $G$
contains no configuration $H_{15}$ and by $R6$. If $f_3(v)=2$,
$n_3(v)=0$ and $f_3^{\frac{4}{3}}(v)=0$, then we
have that $w'(v)\geq 2-1\times2=0$ by $R6$.\\

\textbf{Case $1.2$} $f_3(v)=1$. Then $f_4(v)\leq2$ by
Lemma~\ref{lemma1}.

\textbf{Case $1.2.1$} $f_4(v)=2$.

If $f_4(v)=2$ and the $3$-face incident with $v$ is a
$(3,4,4^+)$-face, then $n_3(v)=1$ and $n_3^{\frac{1}{2}}(v)=0$ for
the reason that $G$ contains no configurations $H_{13}$, $H_{14}$
and by $R2$, $R1$. Thus $w'(v)\geq
2-\frac{2}{3}-\frac{1}{2}\times2-\frac{1}{3}=0$ by $R6$, $R5$,
$R3$.

If $f_4(v)=2$ and the $3$-face incident with $v$ is a
$(4,4,4)$-face, then $n_3(v)=0$, $f_4^{\frac{1}{2}}(v)=0$ for the
reason that $G$ contains no configuration $H_{16}$ and $R5$. Thus
$w'(v)\geq 2-\frac{4}{3}-\frac{1}{3}\times2=0$ by $R6$, $R5$.

If $f_4(v)=2$ and the $3$-face is a $(4,4,5^+)$- or
$(4,5,5^+)$-face, then $n_3(v)\leq1$ for the reason that $G$
contains no configuration $H_{17}$. First, we assume $n_3(v)=1$.
Since $G$ contains no configurations $H_{18}$, $H_{19}$ and by
$R5$, we have that $f_4^{\frac{1}{2}}(v)=0$. So $w'(v)\geq
2-1-\frac{1}{3}\times2-\frac{1}{3}=0$ by $R6$, $R5$ and $R3$. Now,
we assume that $n_3(v)=0$. Thus $w'(v)> 2-1-\frac{1}{2}\times2=0$
by $R6$ and $R5$.

If $f_4(v)=2$ and the $3$-face is a $(4,6^+,6^+)$-face, then
$f_3^{\frac{2}{3}}(v)=0$ by $R6$. Furthermore, as $n_3(v)\leq2$,
we have that $w'(v)>
2-\frac{1}{2}\times2-\frac{1}{3}-\frac{1}{2}=\frac{1}{6}>0$ by
$R5$, $R3$ and $R2$. By Fact~\ref{fact11}, this concludes the case
where $f_4(v)=2$.\\

\textbf{Case $1.2.2$} $f_4(v)=1$.

If $f_4(v)=1$ and the $3$-face is a $(4,4,4)$-face, then
$n_3(v)=0$ for the reason that $G$ contains no configurations
$H_{16}$, $H_{20}$ and $H_{21}$. Thus $w'(v)\geq
2-\frac{4}{3}-\frac{1}{2}=\frac{1}{6}>0$ by $R6$ and $R5$.

If $f_4(v)=1$ and the $3$-face is a $(4,4,5^+)$- or
$(4,5,5^+)$-face, then $n_3(v)=1$ and $n_3^{\frac{1}{2}}(v)=0$ for
the reason that $G$ contains no configurations $H_{17}$, $H_{22}$
and by $R2$, $R1$. Thus
$w'(v)\geq2-1-\frac{1}{2}-\frac{1}{3}=\frac{1}{6}>0$ by $R6$,
$R5$, $R3$.

If $f_4(v)=1$ and the $3$-face is a $(3,4,4^+)$-face, then
$n_3(v)=1$ and $n_3^{\frac{1}{2}}(v)=0$ for the reason that $G$
contains no configurations $H_{13}$ and $H_{14}$ and by $R2$,
$R1$. Thus
$w'(v)\geq2-\frac{2}{3}-\frac{1}{2}-\frac{1}{3}=\frac{1}{2}>0$ by
$R6$, $R5$ and $R3$.

If $f_4(v)=1$ and the $3$-face is a $(4,6^+,6^+)$-face, then
$f_3^{\frac{2}{3}}(v)=0$, $n_3(v)\leq2$ by Fact~\ref{fact12} and
$R6$. Thus $w'(v)\geq
2-\frac{1}{2}-\frac{1}{2}\times2=\frac{1}{2}>0$ by $R5$ and $R2$.
By fact~\ref{fact11}, this completes this subcase.\\

\textbf{Case $1.2.3$} $f_4(v)=0$.

If the $3$-face is a $(4,4,4)$-face, then $n_3(v)\leq1$ and
$n_3^{\frac{1}{2}}(v)=0$ for the reason that $G$ contains no
configurations $H_{17}$, $H_{22}$ and by $R2$, $R1$. Thus
$w'(v)\geq 2-\frac{4}{3}-\frac{1}{3}=\frac{1}{3}>0$ by $R6$ and
$R3$.

If the $3$-face is a $(3,4,4^+)$- or $(4,4^+,5^+)$-face, then
$n_3(v)\leq2$ for the reason that $G$ contains no configuration
$H_{13}$. Thus $w'(v)\geq 2-1-1=0$ by Fact~\ref{fact13}, $R5$ and
$R1$. By
Fact~\ref{fact11}, this conclude the subcase $f_4(v)=0$.\\

\textbf{Case $1.3$} $f_3(v)=0$. Then $f_4(v)\leq2$ by
Lemma~\ref{lemma1}.

If $f_4(v)=2$, then $n_3(v)\leq2$ by Fact~\ref{fact12}. Thus
$w'(v)\geq2-\frac{1}{2}\times2-1=0$ by Fact~\ref{fact13}, $R6$ and
$R1$. If $f_4(v)=1$, then $n_3(v)\leq3$ by Fact~\ref{fact12}. Thus
$w'(v)\geq2-\frac{1}{2}-\frac{1}{2}-\frac{1}{3}\times2=\frac{1}{3}>0$
by $R5$, $R2$ and $R3$. Otherwise, $f_4(v)=0$, then $n_3(v)\leq4$.
Thus $w'(v)\geq 2-\frac{1}{2}-\frac{1}{3}\times3=\frac{1}{2}>0$
by Fact~\ref{fact13} and $R2$ and $R3$.\\

\textbf{Suppose $d(v)=5$}. Then $w(v)=5$.

\textbf{Case $1.4$} $v$ is weakly incident with a bad face.
Clearly, $f_3(v)\leq2$. Furthermore, if $f_3(v)=2$, then
$f_4(v)\leq1$ by Lemma~\ref{lemma1}.

If $f_3(v)=2$ and $f_4(v)=1$, then one of the two $3$-faces must
be adjacent to a bad face which is weakly incident with $v$ by
Lemma~\ref{lemma1}. Obviously, it is a $(3,5,3^+)$-face. In
detail, it is a special face (i.e. a $(3,5,3)$-face) or a
$(3,5,4^+)$-face. Since $G$ contains no configuration $H_{23}$,
the other $3$-face is neither a $(4,4,5)$- nor a $(4,5,5)$-face.
Thus
$w'(v)\geq5-\frac{11}{6}-\frac{4}{3}-\frac{2}{3}-\frac{1}{2}=\frac{2}{3}>0$
or
$w'(v)\geq5-\frac{11}{6}-\frac{4}{3}-\frac{2}{3}-\frac{1}{2}-\frac{1}{2}=\frac{1}{6}>0$
by Fact~\ref{fact14}, $R6$, $R5$, $R7$ and $R2$.

If $f_3(v)=2$ and $f_4(v)=0$, we have that
$w'(v)\geq5-2\times2-\frac{1}{2}-\frac{1}{2}=0$ by $R6$, $R2$ and
$R7$.

If $f_3(v)\leq1$, then $f_4(v)\leq2$. We have that
$w'(v)\geq5-2-\frac{2}{3}\times2-\frac{1}{2}-\frac{1}{2}=\frac{2}{3}>0$
by $R6$, $R5$, $R7$ and $R2$.\\

\textbf{Case $1.5$} $v$ is not weakly incident with a bad face.
Clearly, $f_3(v)\leq2$.

\textbf{Case $1.5.1$} $f_3(v)=2$. Then $f_4(v)\leq1$.

If both of the $3$-faces are $(4,4,5)$- or $(4,5,5)$-faces, then
$n_3(v)=0$ for the reason that $G$ contains no configuration
$H_{24}$. Thus $w'(v)\geq5-2\times2-\frac{2}{3}=\frac{1}{3}>0$ by
$R6$ and $R5$.

If only one of the $3$-faces is a $(4,4,5)$- or $(4,5,5)$-face,
then the other $3$-face is not a $(3,5,3^+)$-face for the reason
that $G$ contains no configuration $H_{23}$. Thus $n_3(v)\leq1$
and $n_3^{1}=0$ by the Fact~\ref{fact13}. We have that
$w'(v)\geq5-2-\frac{4}{3}-\frac{2}{3}-\frac{1}{2}=\frac{1}{2}>0$
by $R6$, $R5$ and $R2$.

If both of the $3$-faces are $(3,5,3^+)$-faces, then $n_3(v)=2$,
$n_3^{\frac{1}{2}}(v)=0$ for the reason that $G$ contains no
configurations $H_{25}$, $H_{26}$ and by $R2$, $R1$. Thus
$w'(v)\geq5-\frac{11}{6}\times2-\frac{2}{3}-\frac{1}{3}\times2=0$
by $R6$, $R5$ and $R3$.

If only one of the $3$-faces is a $(3,5,3^+)$-face, then
$n_3(v)\leq2$, $n_3^{\frac{1}{2}}(v)\leq1$ for the reason that $G$
contains no configurations $H_{25}$ and $H_{26}$ and by $R2$,
$R1$. Thus
$w'(v)\geq5-\frac{11}{6}-\frac{4}{3}-\frac{2}{3}-\frac{1}{3}-\frac{1}{2}=\frac{1}{3}>0$
by $R6$, $R5$, $R3$ and $R2$.

If any of the $3$-faces does not belong to $(3,5,3^+)$-,
$(4,4,5)$- and $(4,5,5)$-faces, then $n_3(v)\leq1$. Thus
$w'(v)\geq5-\frac{4}{3}\times2-\frac{2}{3}-1=\frac{2}{3}>0$ by
$R6$, $R5$ and $R1$.\\

\textbf{Case $1.5.2$} $f_3(v)=1$. Then $f_4(v)\leq2$,
$n_3(v)\leq4$ ($v$ could be adjacent to five $3$-vertices, but at
most four of them are simple) by Lemma~\ref{lemma1}. Clearly,
$w'(v)\geq5-2-\frac{2}{3}\times2-\frac{1}{2}-\frac{1}{3}\times3=\frac{1}{6}>0$
by Fact~\ref{fact13}, $R6$, $R5$, $R2$ and $R3$.\\

\textbf{Case $1.5.3$} $f_3(v)=0$. Then $f_4(v)\leq2$,
$n_3(v)\leq5$ by Lemma~\ref{lemma1}. Clearly,
$w'(v)\geq5-\frac{2}{3}\times2-\frac{1}{2}-\frac{1}{3}\times4=\frac{11}{6}>0$
by Fact~\ref{fact13}, $R5$, $R2$ and $R3$.\\

\textbf{Suppose $d(v)=6$}. Then $w(v)=8$.

First, we assume that $v$ is weakly incident with a bad face.
Clearly, $f_3(v)\leq3$. If $f_3(v)=3$, then $f_4(v)=0$ by
Lemma~\ref{lemma1}. Clearly,
$w'(v)\geq8-2\times3-\frac{1}{2}-\frac{1}{2}=1>0$ by
Fact~\ref{fact14}, $R6$, $R7$ and $R2$. If $f_3(v)\leq2$, then
$f_4(v)\leq2$. Clearly,
$w'(v)\geq8-2\times2-1\times2-\frac{1}{2}-\frac{1}{2}=1>0$ by
Fact~\ref{fact14}, $R6$, $R5$, $R7$ and $R2$.\\

Now we assume that $v$ is not weakly incident with a bad face.
Clearly, $f_3(v)\leq3$. If $f_3(v)=3$, then $f_4(v)=0$,
$n_3(v)\leq3$ (a $3$-face is incident with at most one simple
$3$-vertex) by Lemma~\ref{lemma1}. Thus $w'(v)\geq
8-2\times3-\frac{1}{2}-\frac{1}{3}\times2=\frac{5}{6}>0$ by
Fact~\ref{fact13}, $R6$, $R2$ and $R3$. If $f_3(v)=2$. Then
$f_4(v)\leq2$, $n_3(v)\leq4$ by Lemma~\ref{lemma1}. Thus
$w'(v)\geq
8-2\times2-1\times2-\frac{1}{3}\times3-\frac{1}{2}=\frac{1}{2}>0$
by $R6$, $R5$, $R3$ and $R2$. If $f_3(v)\leq1$, then
$f_4(v)\leq3$, $n_3(v)\leq6$ by Lemma~\ref{lemma1}. Clearly,
$w'(v)> 8-2-1\times3-\frac{1}{3}\times5-\frac{1}{2}=\frac{5}{6}>0$
by $R6$, $R5$, $R3$ and $R2$.\\

\textbf{Suppose $d(v)=7$}. Then $w(v)=11$.

First, we assume that $v$ is weakly incident with a bad face.
Clearly, $f_3(v)\leq3$ by Lemma~\ref{lemma1}. Furthermore,
$f_3^{3}(v)\leq1$ for the reason that $G$ contains no
configuration $H_{27}$ and by $R6$. If $f_3(v)=3$, then
$f_4(v)\leq1$ by Lemma~\ref{lemma1}. Clearly,
$w'(v)\geq11-3-2\times2-\frac{4}{3}-\frac{1}{2}-\frac{1}{2}=\frac{5}{3}>0$
by $R6$, $R5$, $R7$ and $R2$. If $f_3(v)\leq2$, then
$f_4(v)\leq3$. Clearly,
$w'(v)\geq11-3-2-\frac{4}{3}\times3-\frac{1}{2}-\frac{1}{2}=1>0$
by $R6$, $R5$, $R7$ and $R2$.\\

Now we assume that $v$ is not weakly incident with a bad face.
Clearly, we have $f_3(v)\leq3$. Since $G$ contains no
configuration $H_{27}$, there exists at most one $(3,4,7)$-face
which is incident with $v$. If $f_3(v)=3$, then $f_4(v)\leq1$,
$n_3(v)\leq4$ by Lemma~\ref{lemma1}. Thus
$w'(v)\geq11-3-2\times2-\frac{4}{3}-\frac{1}{3}\times
3-\frac{1}{2}=\frac{7}{6}>0$ by Fact~\ref{fact13}, $R6$, $R5$,
$R3$ and $R2$. If $f_3(v)=2$, then $f_4(v)\leq3$, $n_3(v)\leq5$ by
Lemma~\ref{lemma1}. Thus
$w'(v)\geq11-3-2-\frac{4}{3}\times3-\frac{1}{3}\times4-\frac{1}{2}=\frac{1}{6}>0$
by Fact~\ref{fact13}, $R6$, $R5$, $R3$ and $R2$. If $f_3(v)\leq1$,
then $f_4(v)\leq3$, $n_3(v)\leq7$ by Lemma~\ref{lemma1}. Thus
$w'(v)\geq11-3-\frac{4}{3}\times3-\frac{1}{3}\times6-\frac{1}{2}=\frac{3}{2}>0$
by Fact~\ref{fact13}, $R6$, $R5$, $R3$ and $R2$.\\

\textbf{Suppose $d(v)\geq8$}. Then $w(v)=3d(v)-10$.

In any case, whether $v$ is weakly incident with a bad face or
not, we have
\begin{eqnarray}\label{formula20}
f_3(v)+f_4(v)\leq \frac{3}{4}d(v)
\end{eqnarray} by Lemma~\ref{lemma1}. Moreover,
\begin{eqnarray}\label{formula2}
f_3^{3}(v)\leq1
\end{eqnarray}
for the reason that $G$ contains no configuration $H_{27}$ and by
$R6$. Since a $3$-face has at most one simple $3$-vertex,
\begin{eqnarray}\label{formula3}
n_3(v)\leq f_3(v)+d(v)-2f_3(v)=d(v)-f_3(v).
\end{eqnarray}
It follows from (\ref{formula20}) and (\ref{formula3}) that
$f_4(v)\leq \frac{3}{4}d(v)-f_3(v)$ and $n_3(v)\leq d(v)-f_3(v)$,
respectively. Thus $w'(v)\geq
3d(v)-10-3-2(f_3(v)-1)-\frac{4}{3}f_4(v)-\frac{1}{2}-\frac{1}{3}(n_3(v)-1)-\frac{1}{2}\geq
3d(v)-10-3-2f_3(v)+2-d(v)+\frac{4}{3}f_3(v)-\frac{1}{2}-\frac{1}{3}d(v)+\frac{1}{3}f_3(v)+\frac{1}{3}-\frac{1}{2}=
\frac{5}{3}d(v)-\frac{1}{3}f_3(v)-\frac{70}{6}$ by $R6$, $R5$,
$R2$, $R3$ and $R7$. Since
\begin{displaymath}
f_3(v)\leq\frac{1}{2}d(v),
\end{displaymath}
we obtain $w'(v)\geq \frac{3}{2}d(v)-\frac{70}{6}\geq
\frac{1}{3}>0$.\\

Now we consider $f\in F(G)$. \textbf{Suppose $d(f)=3$}. Then
$w(f)=-4$. By Fact~\ref{fact11}, we only discuss the following
situations. If $f$ is a special face $(3,3,5^+)$-face, then we
have that $w'(f)\geq -4+\frac{11}{6}+\frac{1}{3}=-\frac{11}{6}>-2$
by Lemma~\ref{lemma01}, $R6$ and $R4$. If $f$ is a $(3,4,4)$-,
$(3,4,5)$- or $(3,4,6)$-face, we have that $w'(f)\geq -4+
\frac{2}{3}\times2+\frac{1}{3}=-\frac{7}{3}$ by
Lemma~\ref{lemma01}, $R6$ and $R4$. If $f$ is a $(3,4,7^+)$-face,
then $w'(f)\geq -4+\frac{2}{3}+3+\frac{1}{3}=0$  by
Lemma~\ref{lemma01}, $R6$ and $R4$. If $f$ is a
$(3,5^+,5^+)$-face, then $w'(f)\geq
-4+\frac{11}{6}\times2+\frac{1}{3}=0$  by Lemma~\ref{lemma01},
$R6$ and $R4$. If $f$ is a $(4,4,4)$-face, then $w'(f)\geq
-4+\frac{4}{3}\times3=0$ by $R6$. If $f$ is a $(4,4,5^+)$-face,
then $w'(f)\geq -4+1\times2+2=0$ by $R6$. If $f$ is a
$(4,5,5^+)$-face, we have $w'(f)\geq -4+1+2\times2=1>0$ by $R6$.
If $f$ is a $(4,6^+,6^+)$-face, then $w'(f)\geq -4+2\times2=0$ by
$R6$. If $f$ is a $(5,5,5)$-face, we have that $w'(f)\geq
-4+\frac{4}{3}\times3=0$ by $R6$. If $f$ is a
$(5^+,5^+,6^+)$-face, we have that $w'(f)\geq -4+1\times2+2=0$ by
$R6$.\\

\textbf{Suppose $d(f)=4$}. Then $w(f)=-2$. If $f$ is a
$(3,3,5,5^+)$-face, then it is a bad face. Thus $w'(f)\geq
-2+\frac{1}{2}\times2+\frac{2}{3}\times2=\frac{1}{3}>0$ by $R5$
and $R7$. If $f$ is a $(3,3,6^+,6^+)$-face, then
$w'(v)\geq-2+1\times2=0$ by $R5$. If $f$ is a $(3,4,4,4)$- or
$(3,4,4,5)$-face, then it is a bad face. Thus $w'(f)\geq
-2+\frac{1}{2}+\frac{1}{2}\times2+\frac{1}{2}=0$ by $R5$ and $R7$.
If $f$ is a $(3,4,4,6)$-face, then it is a bad face. Thus
$w'(f)\geq -2+\frac{1}{2}+\frac{1}{3}\times2+1=\frac{1}{6}>0$ by
$R5$ and $R7$.  If $f$ is a $(3,4,4,7^+)$-face, then we have that
$w'(v)\geq-2+\frac{1}{3}\times2+\frac{4}{3}=0$ by $R5$. If $f$ is
a $(3,4,5,5)$-face, then it is a bad face. Thus $w'(f)\geq
-2+\frac{1}{2}+\frac{1}{2}+\frac{2}{3}\times2=\frac{1}{3}>0$ by
$R5$ and $R7$. If $f$ is a $(3,4,5,6)$-face, then it is a bad
face. Thus $w'(f)\geq
-2+\frac{1}{2}+\frac{1}{3}+\frac{2}{3}+1=\frac{1}{2}>0$ by $R5$
and $R7$. If $f$ is a $(3,4,5,7^+)$-face, then $w'(f)\geq
-2+\frac{1}{3}+\frac{2}{3}+\frac{4}{3}=\frac{1}{3}>0$ by $R5$. If
$f$ is a $(3,4,6^+,6^+)$-face, then $w'(f)\geq
-2+\frac{1}{3}+1\times2=\frac{1}{3}>0$ by $R5$. If $f$ is a
$(3,5^+,5^+,5^+)$-face, then $w'(f)\geq -2+\frac{2}{3}\times3=0$
by $R5$. If $f$ is a $(4,4,4,4^+)$-face, then $w'(f)\geq
-2+\frac{1}{2}\times4=0$ by $R5$. If $f$ is a
$(4^+,4^+,5^+,5^+)$-face, then $w'(f)\geq
-2+\frac{1}{3}\times2+\frac{2}{3}\times2=0$ by $R5$.\\

\textbf{Suppose $d(f)=5$}. Then $w'(f)=w(f)=0$.\\

\textbf{Suppose $d(f)\geq 6$}. Then $w'(f)\geq
w(f)-\frac{1}{3}\times
d(f)=2d(f)-10-\frac{1}{3}\times d(f)\geq0$ by $R4$.\\

From the above discussion, if $x$ is neither a special vertex nor
a special face, then $w'(x)\geq 0$ for each $x\in V(G)\cup F(G)$.
Let $w'_s$ denote the total new charge of the special $3$-vertices
and the special $3$-faces. Since the new charge of the special
$3$-vertices is $-1$ (see the case "$d(v)=3$") and since the new
charge of the special face is at least $-2$ if it is a
$(3,3,5^+)$-face and at least $-\frac{7}{3}$ if it is a
$(3,4,4)$-, a $(3,4,5)$-, or a $(3,4,6)$-face (see the case
"$d(f)=3$"), Claim~\ref{claim1} implies that $w'_s\geq
\min\{-2-1-1, -\frac{7}{3}-1\}=-4$. So we obtain that
\begin{eqnarray}\label{formula4} \sum_{x\in V(G)\cup
F(G)}w'(x)\geq -4,
\end{eqnarray}a contradiction to Equation~\ref{formula1}.

\vspace{0.3cm} \noindent{\bf Case 2.} $\delta(G)=2$ and there are
at most two $2$-vertices in $G$.

Since $G$ contains no structure isomorphic to the configuration
$H_{5}$, the $3$-faces which are incident with $2$-vertices may be
$(2,3,5)$- or $(2,4^+,4^+)$-faces. Since $G$ contains no structure
isomorphic to the configuration $H_{6}$, the $4$-faces which are
incident with $2$-vertices may be $(2,3^-,5^+,5^+)$- or
$(2,4^+,4^+,4^+)$-faces.

The discharging rules are the same as the rules in Case 1 except
for the charge which is given to a $3$- or $4$-face which is
incident with $2$-vertices. For each $v\in V(G)$, if $d(v)\geq4$,
then $v$ gives charge $\frac{2}{3}$ to its incident $(2,x,y)$-face
$f$; and $v$ gives charge $\frac{1}{3}$ to its incident
$(2,x,y,z)$-face $f$ only if the face $f$ is not adjacent to other
$4$-faces which are incident with $v$, otherwise, $v$ gives charge
$\frac{1}{3}$ to only one of the adjacent $4$-faces. Clearly, the
charge which is given to a $(2,x,y)$- (resp. $(2,x,y,z)$)-face is
not greater than that which is given to $(3,x,y)$- (resp.
$(3,x,y,z)$)-faces. For each $v\in V(G)$, the number of $(2,x,y)$-
(resp. $(2,x,y,z)$)-faces which is incident with and accept charge
from $v$ is not greater than that of $(3,x,y)$- (resp.
$(3,x,y,z)$)-faces which is incident with $v$. So we can guarantee
the new charge of each element $x\in V(G)\cup F(G)$ is larger than
or equal to zero except for the special $3$-vertices, the special
$3$-faces, the $2$-vertices and the $3$- or $4$-faces which are
incident with the $2$-vertices. For convenience, let $w'_{t1}$
(resp. $w'_{t2}$) denote the total new charge of one $2$-vertex
(resp. two $2$-vertices) and the faces which are incident with the
$2$-vertex (resp. the two $2$-vertices).\\

\textbf{Suppose that there exists only one $2$-vertex in $G$}. If
the $2$-vertex is incident with one $3$-face, then it will be not
incident with any $4$-face by Lemma~\ref{lemma1}. Since $G$
contains no configuration $H_5$, the $3$-face is a $(2,3^+,5^+)$-
or $(2,4^+,4^+)$-face, thus $w'_{t1}\geq
-4-4+\frac{2}{3}=-\frac{22}{3}$ or $w'_{t1}\geq
-4-4+\frac{2}{3}\times2=-\frac{20}{3}$. If the $2$-vertex is
incident with a $4$-face, then it may be incident with two
$4$-faces. Furthermore, the $4$-face is a $(2,3^+,5^+,5^+)$- or a
$(2,4^+,4^+,4^+)$-face for the reason that $G$ contains no
configuration $H_6$. Clearly, $w'_{t1}\geq
-2-2-4+\frac{1}{3}\times2=-\frac{22}{3}$ or $w'_{t1}\geq
-2-2-4+\frac{1}{3}\times3=-7$. From the above discussion, we
obtain that
\begin{eqnarray}\label{formula5}
w'_{t1}\geq \min\{-7, -\frac{20}{3},
-\frac{22}{3}\}=-\frac{22}{3}.
\end{eqnarray} By~(\ref{formula4}), we have that $\sum_{x\in V(G)\cup F(G)}w'(x)\geq
-4+w'_{t1}\geq-4-\frac{22}{3}=-\frac{34}{3}$, a contradiction to Equation~\ref{formula1}.\\

\textbf{Suppose that there exist two $2$-vertices in $G$}. If the
two $2$-vertices are incident with a same $3$-face, then $f$ is a
$(2,2,5^+)$-face for the reason that $G$ contains no configuration
$H_{5}$. Thus $w'_{t2}\geq-4\times2-4+\frac{2}{3}=-\frac{34}{3}$.
If the two $2$-vertices are incident with a same $4$-face, then
the $4$-face is a $(2,2,5^+,5^+)$-face for the reason that $G$
contains no configuration $H_{6}$. Since each of the two
$2$-vertices may be incident with another $4$-face, we have that
$w'_{t2}\geq -2-2-2-4-4+\frac{1}{3}\times2=-\frac{40}{3}$. If the
two $2$-vertices are not incident with a same face, then the
discussion is similar to the situation when there exists only one
$2$-vertex in $G$. By~(\ref{formula5}), we have
$w'_{t2}\geq-\frac{22}{3}\times2=-\frac{44}{3}$. From the above
discussion, we have $w'_{t2}\geq\min\{-\frac{44}{3},
-\frac{34}{3}, -\frac{40}{3}\}=-\frac{44}{3}$.
By~(\ref{formula4}), we have that $\sum_{x\in V(G)\cup
F(G)}w'(x)\geq -4-\frac{44}{3}=-\frac{56}{3}$, a contradiction to
Equation~\ref{formula1}.

\vspace{0.3cm} \noindent {\bf Case 3.} $\delta(G)=2$ and there are
at least three $2$-vertices in $G$.

Since $G$ contains no configurations $H_{28}$ $\ldots$ $H_{35}$,
$G$ has the following properties.

\vspace{-0.2cm}
\begin{fact}\label{fact31} Any vertex $v$ is adjacent to at most one
$2$-vertex.
\end{fact}

\vspace{-0.4cm}
\begin{fact}\label{fact32} No two $2$-vertices are adjacent to each
other.
\end{fact}

\vspace{-0.4cm}
\begin{fact}\label{fact33} For each $v\in V(G)$ with $d(v)\geq4$, if
$v$ is adjacent to a $2$-vertex, then it is not incident with any
$3$-face that is incident with a $3$-vertex.
\end{fact}

\vspace{-0.4cm}
\begin{fact}\label{fact34} If $v$ is adjacent to a $3$-vertex, then
it is not incident with any $3$-face that is incident with a
$2$-vertex.
\end{fact}

\vspace{-0.4cm}
\begin{fact}\label{fact35} Every $3$-face in $G$ that is incident with
a $2$-vertex is a $(2,6^+,6^+)$-face.
\end{fact}

\vspace{-0.4cm}
\begin{fact}\label{fact36} If a vertex is adjacent to a $2$-vertex,
then it is not adjacent to any $3$-vertex that is adjacent to
another $3^{--}$-vertex.
\end{fact}

\vspace{-0.4cm}
\begin{fact}\label{fact37} There is at most one $2$-vertex which is
adjacent to a $k$-vertex ($3\leq k\leq 4$) in $G$.
\end{fact}

\vspace{-0.4cm}
\begin{fact}\label{fact38} Any $4$-face that is incident with a
$2$-vertex in $G$ is a $(2,3^+,7^+,7^+)$- or
$(2,6^+,6^+,6^+)$-face.
\end{fact}

For convenience, we call a $2$-vertex a $special$ $2$-$vertex$ if
it is adjacent to a $k$-vertex ($3\leq k\leq 4$), otherwise a
$simple$ $2$-$vertex$. By Fact~\ref{fact37}, there is at most one
special $2$-vertex. Let $n_2(v)$ denote the number of simple
$2$-vertices which are adjacent to $v$. Obviously,
$n_2(v)\in\{0,1\}$ by
Fact~\ref{fact31}.\\

Now redistribute the charge according to the following discharging
rules.

For each $x\in V(G)\bigcup F(G)$, if $x$ is neither a $2$-vertex
nor a face which is not incident with any $2$-vertex, then the
discharging rules are the same as those in Case 1. Otherwise, the
following discharging rules are abided.

\begin{itemize}

\item \textbf{$R8$. Transfer $2$ from each $5^+$-vertex to every
adjacent $2$-vertex.}

\item \textbf{$R9$. Transfer $2$ from each $6^+$-vertex to every
incident $3$-face.}

\item \textbf{$R10$. If $f$ is a $4$-face which is incident with a
$2$-vertex and $v$, then $v$ gives $0$ to $f$ if $d(v)=3$, $4$ or
$5$; $\frac{2}{3}$ if $d(v)=6$; $\frac{4}{3}$ if $d(v)\geq 7$.}

\end{itemize}

By Fact~\ref{fact36}, $R1$ and $R2$, we have the following fact.

\vspace{-0.1cm}

\begin{fact}\label{fact39} For each $v\in V(G)$, if $n_2(v)=1$,
then $n_3^{\frac{1}{2}}(v)=0$.
\end{fact}

In the following, let us check the new charge of each element
$x\in V(G)\bigcup F(G)$.

Consider any vertex $v\in V(G)$, suppose $d(v)=2$. Then $w(v)=-4$,
$n_2(v)=0$ by Fact~\ref{fact32}. Since $G$ contains no structure
$H_9$, $v$ is not weakly incident with any bad face. If $v$ is a
simple $2$-vertex, then $w'(v)=-4+2\times2=0$ by $R8$. Otherwise,
$v$ is a special $2$-vertex. We have $w'(v)=w(v)=-4$.

\textbf{Suppose $d(v)\geq3$}. If $n_2(v)=0$, then the discussion
is similar to the one of the corresponding situation in Case 1. In
the following, we only focus on the situation $n_2(v)=1$.

Since $G$ contains no configurations $H_7$ and $H_8$, we have the
following fact.

\begin{fact}\label{fact310}
For each $v\in V(G)$, if $n_2(v)=1$, then $v$ is not weakly
incident with any bad face.
\end{fact}

Suppose $d(v)=3$. By Fact~\ref{fact33}, $v$ is a simple
$3$-vertex. Since $G$ contains no configuration $H_{10}$, $v$ is
adjacent to at least one $5^+$-vertex or is adjacent to at least
two $4^+$-vertices. We have $w'(v)=-1+1=0$ by $R1$,
or $w'(v)=-1+\frac{1}{2}\times2=0$ by $R2$.\\

\textbf{Suppose $d(v)=4$}. Then $w(v)=2$, $f_3(v)\leq1$ by
Fact~\ref{fact35}.

First we assume $f_3(v)=1$. Then $f_4(v)\leq2$. If the $3$-face is
a $(4,4,4)$-face, then $f_4(v)\leq1$ and $n_3(v)=0$ for the reason
that $G$ contains no configuration $H_{16}$, $H_{17}$ and by
Fact~\ref{fact38}. Thus $w'(v)\geq
2-\frac{4}{3}-\frac{1}{2}-0=\frac{1}{6}>0$ by $R6$, $R9$ and
$R10$. Otherwise, if the $3$-face is not a $(4,4,4)$-face, we have
$f_4(v)\leq2$, $f_4^{\frac{1}{2}}(v)\leq1$ and $n_3(v)\leq1$ for
the reason that $G$ contains no $H_{13}$ and by Fact~\ref{fact38},
$R5$, $R10$. Thus $w'(v)\geq
2-1-\frac{1}{2}-\frac{1}{3}=\frac{1}{6}>0$ by Fact~\ref{fact39},
$R6$, $R9$, $R5$ and $R3$.\\

Now we assume that $f_3(v)=0$. Then $f_4(v)\leq2$,
$f_4^{\frac{1}{3}}\leq1$ and $n_3(v)\leq3$ for the reason that $G$
contains no chordal $6$-cycles and by $R10$. Thus $w'(v)\geq
2-\frac{1}{2}-\frac{1}{3}\times3=\frac{1}{2}>0$ by
Fact~\ref{fact39}, $R10$ and $R3$.\\

\textbf{Suppose $d(v)=5$}. Then $w(v)=5$, $f_3(v)\leq2$.

\textbf{Case $3.1$} $f_3(v)=2$. Then $f_4(v)\leq1$ for the reason
that $G$ contains no chordal $4$- and $6$-cycles. By Fact 9 and
Fact 12, the $4$-face which is incident with $v$ is a
$(2,5,7^+,7^+)$-face. Thus $f_4^{\frac{2}{3}}(v)=0$ by $R10$.
Additionally, since $G$ contains no configuration $H_{36}$, we
have that $f_3^{\frac{11}{6}}(v)=0$ by Fact~\ref{fact33} and $R6$.
Thus $w'(v)\geq 5-\frac{4}{3}\times2-2-0=\frac{1}{3}>0$ by $R6$,
$R9$, $R5$ and $R8$.

\textbf{Case $3.2$} $f_3(v)=1$. Since $G$ contains neither chordal
$4$- and $6$-cycles nor configuration $H_{37}$, we have
$f_4(v)\leq3$.

\textbf{Case $3.2.1$} $f_4(v)=3$. Then $n_3(v)\leq1$ by
Fact~\ref{fact33} and Fact~\ref{fact38}. Furthermore, since at
most one $4$-face which is incident with $v$ is not a
$(2,5,7^+,7^+)$-face, we have that $f_4^{\frac{2}{3}}(v)\leq1$ by
$R5$ and $R10$. Thus $w'(v)\geq5-2-\frac{2}{3}-2-\frac{1}{3}=0$ by
Fact~\ref{fact39}, $R6$, $R9$, $R5$, $R8$ and $R3$.

\textbf{Case $3.2.2$} $f_4(v)=2$.

\textbf{Case $3.2.2.1$} The $2$-vertex which is adjacent to $v$ is
not around any of the two $4$-faces. If the $3$-face which is
incident with $v$ is a $(5,6^+,6^+)$-face, then
$f_4^{\frac{2}{3}}(v)\leq2$, $n_3(v)\leq2$ as $G$ contains no
configuration $H_{37}$ and by $R5$, Fact~\ref{fact33}. Thus we
have $w'(v)\geq 5-1-\frac{2}{3}\times2-2-\frac{1}{3}\times2=0$ by
$R6$, $R9$, $R5$, $R8$ and $R3$. Otherwise, the $4$-faces which
are incident with $v$ are both $(5,5^+,5^+,5^+)$-faces as $G$
contains no configuration $H_{37}$. Clearly, $n_3(v)=0$. Thus we
have $w'(v)\geq 5-2-\frac{1}{2}\times2-2=0$ by $R6$, $R9$, $R5$
and $R8$.

\textbf{Case $3.2.2.2$} The $2$-vertex which is adjacent to $v$ is
around one of the two $4$-faces. Then $f_4^{\frac{2}{3}}(v)\leq1$,
$n_3(v)\leq1$ as $G$ contains no configuration $H_{38}$ and by
$R5$, Fact~\ref{fact33}. Thus we have $w'(v)\geq
5-2-\frac{2}{3}-2-\frac{1}{3}=0$ by $R6$, $R9$, $R5$, $R8$ and
$R3$.

\textbf{Case $3.2.2.3$} The $2$-vertex which is adjacent to $v$ is
around the two $4$-faces. Then $f_4^{\frac{1}{2}}(v)=0$,
$n_3(v)\leq1$ by Fact~\ref{fact38}. Thus we have $w'(v)\geq
5-2-2-\frac{1}{3}=\frac{2}{3}>0$ by $R6$, $R9$, $R8$ and $R3$.

\textbf{Case $3.2.3$} $f_4(v)=1$. Then $n_3(v)\leq2$ by
Fact~\ref{fact35} and Fact~\ref{fact36}. If $n_3(v)=2$, then the
$4$-face is adjacent to the $3$-face and the $3$-face is a
$(5,6^+,6^+)$-face as $G$ contains no configuration $H_{37}$ and
$H_{38}$. We have $w'(v)\geq
5-1-\frac{2}{3}-2-\frac{1}{3}\times2=\frac{2}{3}>0$ by $R6$, $R5$,
$R8$ and $R3$. Otherwise, $n_3(v)\leq1$. We have
$w'(v)\geq5-2-\frac{2}{3}-2-\frac{1}{3}=0$ by $R6$, $R5$, $R8$ and
$R3$.

\textbf{Case $3.2.4$} $f_4(v)=0$. Then $n_3(v)\leq2$ by
Fact~\ref{fact35} and Fact~\ref{fact36}. We have
$w'(v)\geq5-2-2-\frac{1}{3}\times2=\frac{1}{3}>0$ by $R6$, $R8$
and $R3$.

\textbf{Case $3.3$} $f_3(v)=0$. Then $f_4(v)\leq3$ for the reason
that $G$ contains no chordal $6$-cycles. Since at most two
$4$-faces which are incident with $v$ are not
$(2,5,7^+,7^+)$-faces, we have $f_4^{\frac{2}{3}}(v)\leq2$ by
$R5$. Furthermore, $n_3(v)\leq4$. Thus $w'(v)\geq
5-\frac{2}{3}\times2-\frac{1}{3}\times4-2=\frac{1}{3}>0$ by $R5$,
$R3$ and $R8$.\\

\textbf{Suppose $d(v)=6$}. Then $w(v)=8$, $f_3(v)\leq3$. If
$f_3(v)=3$, then $f_4(v)=0$, $n_3(v)=0$ for the reason that $G$
contains no chordal $4$- and $6$-cycles and by Fact~\ref{fact33}.
Thus $w'(v)\geq 8-2\times3-2=0$ by $R6$, $R9$ and $R8$. If
$f_3(v)=2$, then $f_4(v)\leq2$, $n_3(v)\leq1$ for the reason that
$G$ contains no chordal $4$- and $6$-cycles and by
Fact~\ref{fact33}, Fact~\ref{fact34}. Since $G$ contains no
configuration $H_{38}$ and by $R10$, we have that $f_4^{1}(v)=0$.
Thus
$w'(v)\geq8-2\times2-\frac{2}{3}\times2-2-\frac{1}{3}=\frac{1}{3}>0$
by Fact~\ref{fact37}, $R6$, $R9$, $R10$, $R8$ and $R3$. If
$f_3(v)\leq1$, then $f_4(v)\leq3$, $n_3(v)\leq5$. Since $G$
contains no configuration $H_{38}$, we have that $f_4^{1}(v)=0$.
Thus
$w'(v)\geq8-2-\frac{2}{3}\times3-2-\frac{1}{3}\times5=\frac{1}{3}>0$
by Fact~\ref{fact39}, $R6$, $R9$, $R10$, $R8$ and $R3$.\\

\textbf{Suppose $d(v)=7$}. Then $w(v)=11$, $f_3(v)\leq3$. By
Fact~\ref{fact33}, there is no $(3,4,7)$-face which is incident
with $v$. If $f_3(v)=3$, then $f_4(v)\leq1$, $n_3(v)=0$ for the
reason that $G$ contains no chordal $4$- and $6$-cycles and by
Fact~\ref{fact33}. Thus $w'(v)\geq
11-2\times3-\frac{4}{3}-2=\frac{5}{3}>0$ by $R6$, $R9$, $R10$ and
$R8$. If $f_3(v)=2$, then $f_4(v)\leq3$, $n_3(v)\leq2$ for the
reason that $G$ contains no chordal $4$- and $6$-cycles and by
Fact~\ref{fact33}, Fact~\ref{fact34}. Thus
$w'(v)\geq11-2\times2-\frac{4}{3}\times3-\frac{1}{3}\times2-2=\frac{1}{3}>0$
by Fact~\ref{fact39}, $R6$, $R9$, $R10$, $R3$ and $R8$. If
$f_3(v)=1$, then $f_4(v)\leq4$, $n_3(v)\leq4$ for the reason that
$G$ contains no chordal $6$-cycles and by Fact~\ref{fact33},
Fact~\ref{fact34}. Thus $w'(v)\geq
11-2-\frac{4}{3}\times4-\frac{1}{3}\times4-2=\frac{1}{3}>0$ by
Fact~\ref{fact39}, $R6$, $R9$, $R10$, $R3$ and $R8$. If
$f_3(v)=0$, then $f_4(v)\leq4$, $n_3(v)\leq6$ for the reason that
$G$ contains no chordal $6$-cycles. Thus $w'(v)\geq
11-\frac{4}{3}\times4-\frac{1}{3}\times6-2=\frac{5}{3}>0$ by
Fact~\ref{fact39}, $R10$, $R3$ and $R8$.\\

\textbf{Suppose $d(v)\geq8$}. Then $w(v)=3d(v)-10$. By
Fact~\ref{fact33}, there is no $(3,4,8^+)$-face which is incident
with $v$. Since $n_3(v)+2f_3(v)+1\leq d(v)$, we have that
\begin{displaymath}
n_3(v)\leq d(v)-2f_3(v)-1.
\end{displaymath}
Since $G$ contains no chordal $4$- and $6$-cycles, we have that
$f_3(v)+f_4(v)\leq \frac{3}{4}d(v)+1$. Thus
\begin{displaymath}
f_4(v)\leq \frac{3}{4}d(v)-f_3(v)+1.
\end{displaymath}
Thus $w'(v)\geq
3d(v)-10-2f_3(v)-\frac{4}{3}f_4(v)-\frac{1}{3}n_3(v)-2\geq
3d(v)-10-2f_3(v)-d(v)+\frac{4}{3}f_3(v))-\frac{4}{3}-\frac{1}{3}d(v)+\frac{2}{3}f_3(v)+\frac{1}{3}-2=
\frac{5}{3}d(v)-\frac{39}{3}\geq\frac{1}{3}\geq0$ by
Fact~\ref{fact39}, $R6$, $R9$, $R10$, $R3$ and $R8$.\\

Consider $f\in F(G)$. \textbf{Suppose $d(f)=3$}. Then $w(f)=-4$
and $n_2(f)\leq 1$. If $n_2(f)=1$, then $f$ is a
$(2,6^+,6^+)$-face by Fact~\ref{fact35}. Thus $w'(f)\geq
-4+2\times2=0$ by $R9$. Otherwise, the discussion is similar to
the corresponding
situation when $d(f)=3$ in Case 1, so it is omitted here.\\

\textbf{Suppose $d(f)=4$}. Then $w(f)=-2$, $n_2(f)\leq1$ by
Fact~\ref{fact32}.

If $n_2(f)=1$. Then $f$ is a $(2,3^+,7^+,7^+)$- or a
$(2,6^+,6^+,6^+)$-face by Fact~\ref{fact38}. Thus $w'(f)\geq
-2+\frac{4}{3}\times2=\frac{2}{3}>0$ or
$w'(v)\geq-2+\frac{2}{3}\times3=0$ by $R10$. If $n_2(f)=0$, then
the discussion is similar to the corresponding situation when
$d(f)=4$ in Case 1, so it is omitted here.\\

\textbf{Suppose $d(f)\geq5$}. Then the discussion is similar to
the corresponding situation in Case 1 and is omitted here.

From the above discussion, we can obtain that $w'(x)\geq 0$ for
each $x\in V(G)\cup F(G)$ that is not a special $3$-vertex, a
special $2$-vertex, nor a special face. From (\ref{formula4}), we
have $w'_s\geq-4-4=-8$ by Claim~\ref{claim1} and
Fact~\ref{fact37}. So we obtain $\sum_{x\in V(G)\cup
F(G)}w'(x)\geq -8$, a contradiction to Equation~\ref{formula1}.

\vspace{0.3cm}

\noindent {\bf Case 4} $\delta(G)=1$.

Now, the $3$-faces in $G$ are $(3^-,5^+,5^+)$-faces or
$(4^+,4^+,4^+)$-faces and any $4$-face that is incident with a
$2$-vertex is a $(2,5^+,5^+,5^+)$-face for the reason that $G$
contains no configurations $H_{39}$ and $H_{40}$. Then there is
neither any special $3$-vertex nor any special face in $G$.

\textbf{Case $4.1$}  There is only one $1$-vertex in $G$.

\textbf{Case $4.1.1$} There are at most two $2$-vertices in $G$.

The discharging rules are the same as the rules in Case 1 except
for the charge which is given to a $3$- or $4$-face which is
incident with $2$-vertices. For each $v\in V(G)$, if $d(v)\geq5$,
then $v$ gives charge $1$ to its incident $(2,x,y)$-face $f$; and
$v$ gives charge $\frac{1}{2}$ to its incident $(2,x,y,z)$-face
$f$ only if the face $f$ is not adjacent to other $4$-faces which
are incident with $v$, otherwise, $v$ gives charge $\frac{1}{2}$
to only one of the adjacent $4$-faces. Clearly, the charge which
is given to a $(2,x,y)$- (resp. $(2,x,y,z)$)-face is not greater
than that which is given to $(3,x,y)$- (resp. $(3,x,y,z)$)-faces.
For each $v\in V(G)$, the number of $(2,x,y)$- (resp.
$(2,x,y,z)$)-faces which is incident with and accept charge from
$v$ is not greater than that of $(3,x,y)$- (resp.
$(3,x,y,z)$)-faces which is incident with $v$. So we can guarantee
the new charge of each element $x\in V(G)\cup F(G)$ is larger than
or equal to zero except for the $2$-vertices and the $3$- or
$4$-faces which are incident with the $2$-vertices. For
convenience, let $w'_{t1}$ (resp. $w'_{t2}$) denote the total new
charge of one $2$-vertex (resp. two $2$-vertices) and the faces
which are incident to the $2$-vertex (resp. the two $2$-vertices).

\textbf{Suppose that there is one $2$-vertex in $G$}. If the
$2$-vertex is incident with one $3$-face, then it will be not
incident with any $4$-face as $G$ contains no chordal $4$-cycles.
Since the $3$-face is a $(2,5^+,5^+)$-face, we have that
$w'_{t1}\geq -4-4+1\times2=-6$. If the $2$-vertex is incident with
some $4$-faces, since each such $4$-face is a
$(2,5^+,5^+,5^+)$-face, we have that $w'_{t1}\geq
-2-2-4+\frac{1}{2}\times4=-6$. From the above discussion, we
obtain that
\begin{eqnarray}\label{formula10}
w'_{t1}\geq -6.
\end{eqnarray} So $\sum_{x\in V(G)\cup F(G)}w'(x)\geq -7+w'_{t1}\geq-7-6\geq-13$ (a $1$-vertex has charge $-7$),
a contradiction to Equation~\ref{formula1}.\\

\textbf{Suppose that there are two $2$-vertices in $G$}. Since the
two $2$-vertices are not incident with a same $3$- or $4$-face, by
(\ref{formula10}), we have that $w'_{t2}\geq-6\times2=-12$. So
$\sum_{x\in V(G)\cup F(G)}w'(x)\geq -7-12=-19$, a contradiction to
Equation~\ref{formula1}.\\

\textbf{Case $4.1.2$} There are at least three $2$-vertices in
$G$. The discharging rules are the same as Case 3. It follows from
the discussion which is the same as the situation in Case 3 that
$\sum_{x\in V(G)\cup F(G)}w'(x)\geq
-7-4=-11$, a contradiction to Equation~\ref{formula1}.\\

\textbf{Case $4.2$}  There are at least two $1$-vertices in $G$.

If there are two $1$-vertices in $G$, then there is neither a
$2$-vertex nor a third $1$-vertex in $G$ for the reason that $G$
contains no configuration $H_{41}$. The discharging rules are the
same as Case 1. It follows from the discussion which is the same
as the situation in Case 1 that $\sum_{x\in V(G)\cup
F(G)}w'(x)\geq -7\times2=-14$, a contradiction to
Equation~\ref{formula1}.
\end{proof}

\begin{lemma}\label{hajs} (\cite{hajs}) Every graph has an equitable
$k$-coloring whenever $k\geq \Delta(G)+1$.
\end{lemma}

\begin{lemma}\label{wang} (\cite{pelsmajer, wang}) Every graph $G$ with maximum degree $\Delta(G)\leq 3$
is equitably $k$-choosable whenever $k\geq\Delta(G)+1$.
\end{lemma}

In the following, let us give the proof of the main theorem.

\begin{theorem}\label{theorem1} If $G$ is a planar graph without chordal $4$- and $6$-cycles,
then $G$ is equitably $k$-colorable where
$k\geq\max\{7,\Delta(G)\}$.
\end{theorem}

\begin{proof} Let $G$ be a counterexample with fewest vertices. If each component of $G$ has
at most four vertices, then $\Delta(G)\leq 3$. Clearly, $G$ is
equitably $k$-colorable by Lemma~\ref{hajs}. Otherwise, there is
at least one component with at least five vertices.

For convenience, we divide all the configurations in Figure 1 and
Figure 2 into two classes according to whether it contains the
vertex which is labelled $x_{k-3}$ or not. A configuration belongs
to $C_1$ if it contains the vertex labelled $x_{k-3}$, otherwise,
it belongs to $C_2$.

Suppose that $G$ has one of the configurations of $C_1$. In the
following, we show how to find a set $S$ in order to apply
Lemma~\ref{jun1}. For convenience, let $S'$ be the set of the
labelled vertices of this configuration. For example, if $G$ has
the configuration $H$ depicted in Figure 1, then let $S'=\{x_k,
x_{k-1}, \cdots, x_{k-4}, x_1\}$. By
Corollary~\ref{cor4degenerate}, $G$ is $4$-degenerate. Thus
starting from $S'$, we can find the remaining unspecified vertices
to obtain the set $S$ of Lemma~\ref{jun1} from highest to lowest
indices by choosing a vertex with the minimum degree in the graph
obtained from $G$ by deleting the vertices already being chosen
for $S$ at each step. By the minimality of $G$, we have $G-S$ is
equitably $k$-colorable. By Lemma~\ref{jun1}, we can obtain that
$G$ is equitably $k$-colorable, a contradiction.

Thus $G$ has a configuration of $C_2$ and $\delta(G)\leq3$ by
Lemma~\ref{lemma2}. Similarly, let $S''$ be the set of the
labelled vertices of this configuration, in which the vertices are
labelled as they are in Figure 2. Let $G'=G-S''$. If there exists
a vertex $v\in V(G')$ such that $d_{G'}(v)\leq3$ or there exists a
vertex $u\in \{x_1, x_2, x_3\}\cap S''$ such that $d_G(u)\leq4$,
then we label $v$ or $u$ with $x_{k-3}$ and let
$S'''=S''\cup\{x_{k-3}\}$. By Corollary~\ref{cor4degenerate}, $G$
is $4$-degenerate. Now starting from $S'''$, we can find the
remaining unspecified vertices to obtain the set $S$ of
Lemma~\ref{jun1} from highest to lowest indices by choosing a
vertex with the minimum degree in the graph obtained from $G$ by
deleting the vertices already being chosen for $S$ at each step.
By the minimality of $G$, we have $G-S$ is equitably
$k$-colorable. By Lemma~\ref{jun1}, we can obtain that $G$ is
equitably $k$-colorable, a contradiction.

Thus $\delta(G')\geq4$ and $d_G(v)\geq5$ for each vertex
$v\in\{x_1, x_2, x_3\}\cap S''$. Clearly, it follows the Fact.
\vspace{-0.3cm}

\begin{fact}\label{fact01} For each $x\in V(H')-\{x_k, x_{k-1},
x_{k-2}\}$, we have that $d_G(x)\geq5$ where $H'\in C_2$.
\end{fact}

Now we can easily get that $G$ has only one configuration that
belongs to $C_2$. Otherwise, $\delta(G')\leq3$. Additionally, by
Lemma~\ref{lemma2}, $G'$ contains the configuration $H$ of Figure
1. If $G$ does not contain the configuration $H_{41}$, then by
Fact~\ref{fact01}, at most one $1$-vertex, at most two
$3^-$-vertices and at most one special face can exist in $G$
simultaneously, i.e. $G$ contains the configuration $H_{39}$. Let
us now show a self-contradictory conclusion by a discharging
procedure. The discharging rules are the same as Case 1 in
Lemma~\ref{lemma3}. Clearly, we can guarantee that the new charge
of each face other than the special face, and each vertex $v\in
V(G)$ with $d(v)\geq 4$ is larger than or equal to zero. Hence
$\sum_{x\in V(G)\cup F(G)}w'(x)\geq -7+-4\times2-4=-19$, a
contradiction to $\sum_{x\in V(G)\cup F(G)}w(x)=-20$.

Thus $G$ contains the configuration $H_{41}$. Additionally, from
the above discussion, we know $G$ has no configuration $H$, and
$G'$ has the configuration $H$ in Figure 1. It is clear that one
of the vertices $\{x_k, x_{k-1}, x_{k-2}, x_1\}$ of configuration
$H_{41}$ in Figure 2 must be adjacent to one of the vertices
$\{x_k, x_{k-1}, x_{k-2}\}$ of configuration $H$ in Figure 1. It
is not difficult to find a set $\bar{S}$, starting from which, we
can find the remaining unspecified vertices in $S$ of
Lemma~\ref{jun1} from highest to lowest indices by choosing a
vertex with the minimum degree in the graph obtained from $G$ by
deleting the vertices already being chosen for $S$ at each step.
By the minimality of $G$, we have that $G-S$ is equitably
$k$-colorable. By Lemma~\ref{jun1}, we have that $G$ is equitably
$k$-colorable, a contradiction. In the following, we give the
detailed steps on how to find the set $\bar{S}$.

For convenience, we use $w_1$, $w_2$, $w_3$ and $w_4$ to denote
the vertices $x_k$, $x_{k-1}$, $x_{k-2}$ and $x_1$ of
configuration $H_{41}$ in Figure 2, respectively, and use $u_1$,
$u_2$ and $u_3$ to denote the vertices $x_k$, $x_{k-1}$ and
$x_{k-2}$ of configuration $H$ in Figure 1, respectively.

If there exists one $1$-vertex which is adjacent to one of the
vertices in $\{u_1, u_2, u_3\}$, then the $1$-vertex only may be
$w_2$ or $w_3$ from the above discussion. Without loss of
generality, we assume $w_2$ and $u'$ are adjacent to $u$ for which
$\{u, u'\}\subset\{u_1, u_2, u_3\}$. Now we label the vertices
$w_2, w_1, w_3, u, u'$ with $x_k, x_{k-1}, x_{k-2}, x_{k-3},
x_{k-4}$, respectively. We choose $\bar{S}=\{x_k, x_{k-1},
x_{k-2}, x_{k-3}, x_{k-4}\}$.

Otherwise, if $w_3$ is adjacent to one of the vertices in $\{u_1,
u_2, u_3, u_4\}$ such that $d_G(w_3)=2$, for convenience, we
assume $w_3$ and $u'$ are adjacent to $u$ for which $\{u,
u'\}\subset\{u_1, u_2, u_3\}$. Now we label the vertices $w_1,
w_2, w_3, u, u', w_4$ with $x_k, x_{k-1}, x_{k-2}, x_{k-3},
x_{k-4}, x_1$, respectively. We choose $\bar{S}=\{x_k, x_{k-1},
x_{k-2}, x_{k-3}, x_{k-4}, x_1\}$.

If $w_4$ is adjacent to one of the vertices in $\{u_1, u_2, u_3,
u_4\}$, for convenience, we assume $w_4$ and $u'$ are adjacent to
$u$ for which $\{u, u'\}\subset\{u_1, u_2, u_3\}$. Now we label
the vertices $w_1, w_2, w_3, u, u', w_4$ with $x_k, x_{k-1},
x_{k-2}, x_{k-3}, x_{k-4}, x_1$, respectively. We choose
$\bar{S}=\{x_k, x_{k-1}, x_{k-2}, x_{k-3}, x_{k-4}, x_1\}$. This
completes the proof of Theorem~\ref{theorem1}.
\end{proof}

\begin{corollary}
Let $G$ be a planar graph without chordal $4$- and $6$-cycles. If
$\Delta(G)\geq 7$, then $\chi_e(G)\leq \Delta(G)$.
\end{corollary}

\begin{corollary}
Let $G$ be a planar graph without chordal $4$- and $6$-cycles. If
$\Delta(G)\geq 7$, then $\chi^*_e(G)\leq \Delta(G)$.
\end{corollary}

\begin{theorem} If $G$ is a planar graph without chordal $4$- and $6$-cycles and
$k\geq \max\{7,\Delta(G)\}$, then $G$ is equitably $k$-choosable.
\end{theorem}

\begin{proof} Let $G$ be a counterexample with the fewest vertices, i.e. $G$ is a
critical graph. If each component of $G$ has at most four
vertices, then $\Delta(G)\leq 3$. So $G$ is equitably
$k$-choosable by Lemma~\ref{wang}. Otherwise, the proof is similar
to the proof of Theorem~\ref{theorem1} by Lemma~\ref{lemma3} and
Lemma~\ref{kostochka1}.
\end{proof}

\begin{corollary}
Let $G$ be a planar graph without chordal $4$- and $6$-cycles. If
$\Delta(G)\geq 7$, then $G$ is equitably $\Delta(G)$-choosable.
\end{corollary}

\section{Remarks and perspective}
Most of the results on equitable and list equitable colorings on
planar graphs are restricted to $3$-degenerate graphs. In this
paper, we confirm the Conjecture~\ref{chenconj} and
Conjecture~\ref{kostochconj2} for the planar graphs without
chordal $4$- and $6$-cycles which are not necessarily
$3$-degenerate. Can a similar conclusion be drawn for
$4$-degenerate graphs and ordinary planar graphs?

\acknowledgements We would like to thank the referees for
providing some very helpful suggestions for revising this paper.

\nocite{*}
\bibliographystyle{abbrvnat}
\bibliography{sample-dmtcs}

\begin{thebibliography}{34}
\providecommand{\natexlab}[1]{#1}
\providecommand{\url}[1]{\texttt{#1}}
\expandafter\ifx\csname urlstyle\endcsname\relax
  \providecommand{\doi}[1]{doi: #1}\else
  \providecommand{\doi}{doi: \begingroup \urlstyle{rm}\Url}\fi

\bibitem[Bondy and Murty(1976)]{bondy}
J.~A. Bondy and U.~S.~R. Murty.
\newblock \emph{Graph Theory with Applications}.
\newblock North-Holland, New York, 1976.

\bibitem[Borodin(1996)]{borodin}
O.~V. Borodin.
\newblock Structural properties of plane graphs without adjacent triangles and
  an application to 3-colorings.
\newblock \emph{J. Graph Theory}, 12:\penalty0 183--186, 1996.

\bibitem[Chen and Lih(1994)]{chenl}
B.~L. Chen and K.~W. Lih.
\newblock Equitable coloring of trees.
\newblock \emph{J. Combin. Theory Ser. B}, 61:\penalty0 83--87, 1994.

\bibitem[Chen et~al.(1994)Chen, Lih, and Wu]{chenlw}
B.~L. Chen, K.~W. Lih, and P.~L. Wu.
\newblock Equitable coloring and the maximum degree.
\newblock \emph{European J. Combin.}, 15:\penalty0 443--447, 1994.

\bibitem[Chen and Yen(2012)]{chenY}
B.~L. Chen and C.~H. Yen.
\newblock Equitable $\delta$-coloring of graphs.
\newblock \emph{Discrete Mathematics}, 312\penalty0 (9):\penalty0 1512--1517,
  2012.

\bibitem[Dong et~al.(2013)Dong, Li, and Wang]{dong3}
A.~J. Dong, G.~J. Li, and G.~H. Wang.
\newblock Equitable and list equitable colorings of planar graphs without
  4-cycles.
\newblock \emph{Discrete Mathematics}, 313\penalty0 (15):\penalty0 1610--1619,
  2013.

\bibitem[Dong et~al.(2012{\natexlab{a}})Dong, Tan, Zhang, and Li]{dong1}
A.~J. Dong, X.~Tan, X.~Zhang, and G.~J. Li.
\newblock Equitable coloring and equitable choosability of planar graphs
  without 6- and 7-cycles.
\newblock \emph{Ars Combinatoria}, 103:\penalty0 333--352, 2012{\natexlab{a}}.

\bibitem[Dong et~al.(2012{\natexlab{b}})Dong, Zhang, and Li]{dong2}
A.~J. Dong, X.~Zhang, and G.~J. Li.
\newblock Equitable coloring and equitable choosability of planar graphs
  without 5- and 7-cycles.
\newblock \emph{Bulletin of the Malaysian Mathematical Sciences Society},
  35:\penalty0 897--910, 2012{\natexlab{b}}.

\bibitem[Hajnal and Szemer\'{e}di(1970)]{hajs}
A.~Hajnal and E.~Szemer\'{e}di.
\newblock Proof of a conjecture of erd\"{o}s, in: A. r\'{e}nyi, v.t.
  s\'{o}s(eds.), in: Combin theory and its applications.
\newblock \emph{North-Holland, Amsterdam}, II:\penalty0 601--623, 1970.

\bibitem[Janson and Ruci\'{n}ski(2002)]{jans}
S.~Janson and A.~Ruci\'{n}ski.
\newblock The infamous upper tail.
\newblock \emph{Random Structure and Algorithms}, 20:\penalty0 317--342, 2002.

\bibitem[Kierstead and Kostochka(2012)]{kosn3}
H.~A. Kierstead and A.~V. Kostochka.
\newblock Equitable list coloring of graphs with bounded degree.
\newblock \emph{Journal of Graph Theory}, 2012.
\newblock \doi{10.1002/jgt.21710}.

\bibitem[Kostochka(2002)]{kost}
A.~V. Kostochka.
\newblock Equitable colorings of outerplanar graphs.
\newblock \emph{Discrete Mathematics}, 258:\penalty0 373--377, 2002.

\bibitem[Kostochka and Nakprasit(2003)]{kosn1}
A.~V. Kostochka and K.~Nakprasit.
\newblock Equitable colorings of $k$-degenerate graphs.
\newblock \emph{Combin. Probab. Comput.}, 12:\penalty0 53--60, 2003.

\bibitem[Kostochka and Nakprasit(2005)]{kosn2}
A.~V. Kostochka and K.~Nakprasit.
\newblock Equitable $\delta$-colorings of graphs with low average degree.
\newblock \emph{Theoretical Computer Science}, 349:\penalty0 82--91, 2005.

\bibitem[Kostochka et~al.(2003)Kostochka, Pelsmajer, and West]{kostochka}
A.~V. Kostochka, M.~J. Pelsmajer, and D.~B. West.
\newblock A list analogue of equitable coloring.
\newblock \emph{J. Graph Theory}, 47:\penalty0 166--177, 2003.

\bibitem[Li and Bu(2009)]{qiong}
Q.~Li and Y.~H. Bu.
\newblock Equitable list coloring of planar graphs without 4- and 6-cycles.
\newblock \emph{Discrete Mathematics}, 309:\penalty0 280--287, 2009.

\bibitem[Lih and Wu(1996)]{lihw}
K.~W. Lih and P.~L. Wu.
\newblock On equitable coloring of bipartite graphs.
\newblock \emph{Discrete Math.}, 151:\penalty0 155--160, 1996.

\bibitem[Luo et~al.(2010)Luo, Sereni, Stephens, and Yu]{luo}
R.~Luo, J.~S. Sereni, D.~C. Stephens, and G.~Yu.
\newblock Equitable coloring of sparse planar graphs.
\newblock \emph{SIAM Journal on Discrete Mathematics}, \penalty0 (4):\penalty0
  1572--1583, 2010.

\bibitem[Meyer(1973)]{meye}
W.~Meyer.
\newblock Equitable coloring.
\newblock \emph{Amer. Math. Monthly}, 80:\penalty0 920--922, 1973.

\bibitem[Nakprasit(2012{\natexlab{a}})]{Nakprasit}
K.~Nakprasit.
\newblock Equitable colorings of planar graphs with maximum degree at least
  nine.
\newblock \emph{Discrete Mathematics}, 312\penalty0 (5):\penalty0 1019--1024,
  2012{\natexlab{a}}.

\bibitem[Nakprasit(2012{\natexlab{b}})]{Nakprasit2}
K.~Nakprasit.
\newblock Equitable colorings of planar graphs without short cycles.
\newblock \emph{Theoretical Computer Science}, 465:\penalty0 21--27,
  2012{\natexlab{b}}.

\bibitem[Pelsmajer(2004)]{pelsmajer}
M.~F. Pelsmajer.
\newblock Equitable list coloring for graphs of maximum degree 3.
\newblock \emph{J. Graph Theory}, 47:\penalty0 1--8, 2004.

\bibitem[Pemmaraju(2001)]{pemma}
S.~V. Pemmaraju.
\newblock Equitable colorings extend chernoff-hoeffding bounds.
\newblock In \emph{In Proceedings of the 5th International Workshop on
  Randomization and Approximation Techniques in Computer Science(APPROX-RANDOM
  2001)}, pages 285--296, 2001.

\bibitem[Tan(2010)]{tanxiang}
X.~Tan.
\newblock Equitable $\delta$-coloring of planar graphs without 4-cycles.
\newblock In \emph{The Ninth International Symposium on Operations Research and
  Its Applications (ISORA 10)}, pages 400--405, 2010.

\bibitem[Wang and Lih(2004)]{wang}
W.~F. Wang and K.~W. Lih.
\newblock Equitable list coloring of graphs.
\newblock \emph{Taiwanese J. Math.}, 8:\penalty0 747--759, 2004.

\bibitem[Wang and Zhang(2000)]{wanz}
W.~F. Wang and K.~M. Zhang.
\newblock Equitable colorings of line graphs and complete $r$-partite graphs.
\newblock \emph{System Sci. Math. Sci.}, 13:\penalty0 190--194, 2000.

\bibitem[Wu and Wang(2008)]{jianl}
J.~L. Wu and P.~Wang.
\newblock Equtiable coloring planar graphs with large girth.
\newblock \emph{Discrete Mathematics}, 308:\penalty0 985--990, 2008.

\bibitem[Yan and Wang(2014)]{yanwang}
Z.~Yan and W.~Wang.
\newblock Equitable coloring of kronecker products of complete multipartite
  graphs and complete graphs.
\newblock \emph{Discrete Applied Mathematics}, 162:\penalty0 328--333, 2014.

\bibitem[Yap and Zhang(1997)]{yapz1}
H.~P. Yap and Y.~Zhang.
\newblock The equitable $\delta$-coloring conjecture holds for outerplanar
  graphs.
\newblock \emph{Bull. Inst. Math. Acad. Sin.}, 25:\penalty0 143--149, 1997.

\bibitem[Yap and Zhang(1998)]{yapz2}
H.~P. Yap and Y.~Zhang.
\newblock Equitable colorings of planar graphs.
\newblock \emph{J. Combin. Math. Combin. Comput.}, 27:\penalty0 97--105, 1998.

\bibitem[Zhang and Wu(2011)]{zhang}
X.~Zhang and J.~L. Wu.
\newblock On equitable and equitable list coloring of series-parallel graphs.
\newblock \emph{Discrete Mathematics}, 311:\penalty0 800--803, 2011.

\bibitem[Zhu and Bu(2008)]{jun1}
J.~L. Zhu and Y.~H. Bu.
\newblock Equitable list colorings of planar graphs without short cycles.
\newblock \emph{Theoretical Computer Science}, 407:\penalty0 21--28, 2008.

\bibitem[Zhu and Bu(2010)]{jun3}
J.~L. Zhu and Y.~H. Bu.
\newblock Equitable and equitable list colorings of graphs.
\newblock \emph{Theoretical Computer Science}, 411:\penalty0 3873--3876, 2010.

\bibitem[Zhu et~al.(2013)Zhu, Bu, and Min]{jun2}
J.~L. Zhu, Y.~H. Bu, and X.~Min.
\newblock Equitable list-coloring for $c_5$-free plane graphs without adjacent
  triangles.
\newblock \emph{Graphs and Combinatorics}, 31:\penalty0 1--10, 2013.

\end{thebibliography}

\end{document}